\documentclass[11pt,reqno,twoside]{amsart}

\usepackage{graphicx}
\usepackage{lmodern} 
\usepackage[T1]{fontenc}
\usepackage{float}
\usepackage{subfigure}
\usepackage{amsmath,mathdots,amssymb,latexsym,amsthm,mathrsfs,epsfig}
\usepackage{dsfont}
\usepackage{xfrac}
\usepackage{multirow}
\usepackage{breqn}
\usepackage{cite}
\usepackage{enumitem}
\usepackage{environ}
\usepackage[colorlinks=true,linkcolor=blue,citecolor=red,urlcolor=black]{hyperref}
\usepackage{multicol}
\usepackage{mathtools}

\NewEnviron{myequation}{
	\begin{equation}
		\scalebox{1.1}{$\BODY$}
	\end{equation}
}

\usepackage{pifont}

\usepackage{color}
\usepackage{soul}

\renewcommand{\baselinestretch}{1.2}

\makeatletter

\newcommand{\single}{\let\CS=\@currsize\renewcommand{\baselinestretch}{1.1}\tiny\CS}
\newcommand{\singb}{\let\CS=\@currsize\renewcommand{\baselinestretch}{1}\tiny\CS}
\newcommand{\singa}{\let\CS=\@currsize\renewcommand{\baselinestretch}{1.2}\tiny\CS}
\newcommand{\oneandahalfspacing}{\let\CS=\@currsize\renewcommand{\baselinestretch}{1.5}\tiny\CS}
\newcommand{\singlespacing}{\let\CS=\@currsize\renewcommand{\baselinestretch}{1.6}\large\CS}

\newcommand{\bc}{\begin{center}}
\newcommand{\ec}{\end{center}}
\newcommand{\be}{\begin{eqnarray}}
\newcommand{\ee}{\end{eqnarray}}

\newcommand{\Hom}{\operatorname{Hom}}

\newcommand{\Dim}{\operatorname{dim}}
\newcommand{\diag}{\operatorname{diag}}
\newcommand{\Irr}{\operatorname{Irr}}
\newcommand{\Ind}{\operatorname{Ind}}

\newcommand{\Rep}{\operatorname{Rep}}

\newcommand{\beano}{\begin{eqnarray*}}
\newcommand{\eeano}{\end{eqnarray*}}

\newcommand{\ba}{\begin{array}}
\newcommand{\ea}{\end{array}}

\makeatother

\usepackage[
    a4paper,
    textwidth=6.5in,
    textheight=8.9in,
    centering
]{geometry}

\theoremstyle{plain}
\newtheorem{theorem}{Theorem}[section]
\newtheorem{corollary}[theorem]{Corollary}
\newtheorem{lemma}[theorem]{Lemma}

\newtheorem{conj}[theorem]{Conjecture}

\theoremstyle{plain}
\newtheorem{definition}{Definition}[section]

\theoremstyle{remark}
\newtheorem{remark}[theorem]{Remark}

\numberwithin{equation}{section}

\DeclareMathOperator{\D}{D}

\DeclareMathOperator{\GL}{GL}
\DeclareMathOperator{\SL}{SL}
\DeclareMathOperator{\Sp}{Sp}
\DeclareMathOperator{\St}{St}
\DeclareMathOperator{\Nrd}{Nrd}
\DeclareMathOperator{\JL}{JL}

\DeclareMathOperator{\Par}{P}

\DeclareMathOperator{\G}{G}
\DeclareMathOperator{\Ha}{H}

\DeclareMathOperator{\F}{F}

\DeclareMathOperator{\M}{M}
\DeclareMathOperator{\Seg}{Seg}
\DeclareMathOperator{\V}{V}

\DeclareMathOperator{\Ad}{Ad}

\allowdisplaybreaks

\begin{document}

\title[Representations of $\GL_n(\D)$]
{On Representations of $\GL_n(\D)$ admitting a generalized linear period}

\author[Prem Dagar]{Prem Dagar}
\address{Department of Mathematics,
Indian Institute of Science Education and Research Tirupati,
Tirupati 517619, India}
\email{dagarprem5@gmail.com, prem@labs.iisertirupati.ac.in}

\author[Hariom Sharma]{Hariom Sharma}
\address{Department of Mathematics,
Indian Institute of Technology Bombay,
Mumbai 400076, India}
\email{hariomshrma97@gmail.com, hariom@math.iitb.ac.in}



\date{}

\maketitle

\begin{abstract}

      Let $\mathrm{D}$ be a quaternion division algebra over a non-Archimedean local field $\mathrm{F}$ of characteristic zero, and let $\mathrm{G}_n=\mathrm{GL}_n(\mathrm{D})$. We consider the subgroup $\mathrm{H}_{1,n-1}$ of $\mathrm{G}_n$ consisting of block-diagonal matrices of the form $\mathrm{diag}(g_1,g_2)$, where $g_1\in\mathrm{G}_1$ and $g_2\in\mathrm{G}_{n-1}$. For $s\in\mathbb{R}$, let $\chi_s$ be the character of $\mathrm{H}_{1,n-1}$ given by $\chi_s(\mathrm{diag}(g_1,g_2)) = \nu(g_1)^{2s}\nu(g_2)^{-2s}$.  In this article, we classify the irreducible smooth representations of $\mathrm{G}_n$ for $n=3$ and $n=4$ that admit a generalized linear period with respect to $(\mathrm{H}_{1,n-1},\chi_s)$. Motivated by these results, we propose a conjecture that provides a complete classification of  irreducible smooth representations of $\mathrm{G}_n$ admitting such a period for every $n>2$. Assuming this conjecture, we further characterize the representations admitting generalized linear periods in terms of their Langlands parameters. More precisely, we prove that an irreducible smooth representation $\pi$ of $\G_n$ admits a generalized linear period with respect to $(\Ha_{1,n-1},\chi_s)$ if and only if its Langlands parameter $\mathfrak{L}(\pi)$ admits a Weil--Deligne subrepresentation isomorphic to $\mathfrak{L}(\nu^{-2s})$, the $(2n-4)$-dimensional Langlands parameter of the character $\nu^{-2s}$ of $\G_{n-2}$,
      such that the four-dimensional quotient $\mathfrak L(\pi)/\mathfrak L(\nu^{-2s})$ is the Langlands parameter of either the trivial representation $\mathds{1}_2$ of $\G_2$, in which case $s=\pm\frac{n-2}{2};$ 
      or an irreducible infinite-dimensional $\Ha_{1,1}$-distinguished representation of $\G_2$. 
      Furthermore, we verify the Lapid--Prasad conjecture in this setting, which asserts that the $L$-packet of an irreducible representation of $\G_n$ admitting a linear period with respect to $\Ha_{1,n-1}$ remains invariant under the functor  $\rho\mapsto\widetilde{\rho}^{\,\theta}$, where $\widetilde{\rho}$ denotes the contragredient representation and $\rho^\theta$ denotes the twist of $\rho$ by the involution $\theta$ defining the symmetric pair $(\mathrm{G}_n,\mathrm{H}_{1,n-1})$.

\end{abstract}

\makeatletter
\def\@setcopyright{}
\def\@serieslogo{}
\def\serieslogo{}
\makeatother
\maketitle

\section{Introduction}\label{intro}

       Let $G$ be an $\ell$-group and $H$ a closed subgroup of $G$. For a smooth representation $(\pi,\V)$ of $G$ and a character $\chi$ of $H$, the space $\Hom_H(\pi,\chi)$ consists of linear forms $l:\V\to\mathbb{C}$ satisfying $l(\pi(h)v)=\chi(h)l(v)$ for all $h\in H$ and $v\in\V$.  For a smooth representation $\pi$ of $G$, we say that $\pi$ is $(H,\chi)$-distinguished if $\Hom_H(\pi,\chi)\neq 0$ and simply $H$-distinguished if $\chi$ is trivial. The distinction problem asks for a characterization of the $(H,\chi)$-distinguished spectrum. Three questions arise: which representations $\pi$ admit a nonzero $(H,\chi)$-equivariant functional; what is the dimension of the space $ \operatorname{Hom}_{H}(\pi, \chi)$; and in what manner $(H,\chi)$-distinction is reflected in the arithmetic invariants attached to $\pi$, namely its Langlands parameter $\mathfrak{L}(\pi)$,  $L$-functions, and Arthur packets. Questions of this type constitute the local and global counterparts of the relative Langlands program, in which automorphic representations are studied by means of $\chi$-twisted periods attached to a spherical variety $X = H \backslash G$, rather than in terms of the group $G$ alone.

      We next formulate the distinction problem in our setting. Let $\F$ be a non-Archimedean local field of characteristic zero with finite residue field, and let $\D$ be the quaternion division algebra over $\F$. 
      For each $n \in \mathbb{N}$, set $\G_n=\GL_n(\D)$. We adopt the convention that $\G_0$ is the trivial group. Define the character $\nu$ of $\G_n$ by  $\nu(g)=\left|\Nrd_{\D/\F}(g)\right|_{\F}$, for $g\in\G_n$, where $\Nrd_{\D/\F}:\G_n\to\F^\times$ denotes the reduced norm and $|\cdot|_{\F}$ the normalized absolute value on $\F$. Let $\mathds{1}_n$ denote the trivial character of $\G_n$. For $n=1$, we have $\nu_{\mathds{1}_1}=\nu^2$, where $\nu_{\mathds{1}_1}$ is the character of $\G_1$ associated to the trivial character $\mathds{1}_1$ introduced in Section~\ref{ZLC}.
      Let $p$ and $q$ be nonnegative integers with $p+q=n$. Consider the subgroup $\Ha_{p,q}$ of $\G_n$ given by $$\Ha_{p,q}=\{\diag(g_1,g_2):g_1\in\G_p,\ g_2\in\G_q\}.$$ For $s\in\mathbb{R}$, define a character $\chi_s$ of $\Ha_{p,q}$ by $$\chi_s(\diag(g_1,g_2))=\nu(g_1)^{2s}\nu(g_2)^{-2s}.$$ 
      A smooth representation $\pi$ of $\G_n$ is said to admit a generalized linear period with respect to $(\Ha_{p,q},\chi_s)$ if $\Hom_{\Ha_{p,q}}(\pi,\chi_s)\neq0$, and simply a linear period with respect to $\Ha_{p,q}$ if $s=0$.
      The study of such periods naturally leads to questions concerning their existence and uniqueness. For $\GL_n(\F)$, Jacquet and Rallis~\cite[Theorem~1.1]{jacquet1996uniqueness} established the uniqueness of linear periods. Chen and Sun~\cite[Theorem~B]{chen2020uniqueness} subsequently proved a twisted multiplicity-one result for $\GL_{2n}(\F)$ with respect to $\GL_n(\F)\times\GL_n(\F)$ for all but finitely many characters. 
      
      The question of determining the existence of linear periods has also been studied extensively by several authors. More recently, Yang~\cite[Theorem~1.2]{yang2022linear} classified the unitary representations of $\GL_{2n}(\F)$ admitting linear periods with respect to $\GL_n(\F)\times\GL_n(\F)$, while Sharma~\cite[Theorem~1]{hariom} classified the irreducible smooth representations of $\GL_4(\F)$ admitting linear periods with respect to $\GL_2(\F)\times\GL_2(\F)$. We refer the interested reader to \cite{FL,Hak,HM,Mat1,Mat5,Mat2,MatJNT,Mat3 } and references therein for a detailed account of the results on linear periods.

      The study of linear periods for representations of \(\G_n\) has witnessed significant progress in recent years. The multiplicity-one property for linear periods of irreducible smooth representations of \(\G_n\) was proved by Anandavardhanan \emph{et al.}~\cite[Theorem~2.5]{anandavardhanan2024sign}. In recent work, the authors and Verma \cite{DSV2026} continued the study of linear periods for \(\G_n\) with respect to \(\Ha_{1,n-1}\), obtaining a classification of representations admitting such periods for \(n=3,4\) and proposing a conjectural description for general \(n\). According to this conjecture, such representations are precisely the trivial representation and representations parabolically induced from the trivial representation of $\G_{n-2}$ and an infinite-dimensional $\Ha_{1,1}$-distinguished representation of $\G_2$. This description suggests that a similar pattern might persist for generalized linear periods. Surprisingly, the generalized setting reveals a richer structure, with additional families of representations occurring beyond those predicted by the linear periods case.
      In this direction, Lu~\cite{Lu26} recently investigated generalized linear periods for the pair \((\Ha_{p,q},\chi)\), with \(p+q=n\) and \(\chi\) a character of \(\Ha_{p,q}\). Motivated by these developments and the broader perspective provided by the twisted GGP conjecture~\cite{GGP23}, we propose a conjectural classification of irreducible smooth representations of $\G_n$ for $n>2$ admitting generalized linear periods with respect to $(\Ha_{1,n-1},\chi_s)$.

\begin{conj}\label{conj}
      Let $n>2$. The irreducible smooth $\G_n$-representations admitting a generalized linear period with respect to $(\Ha_{1,n-1},\chi_s)$ are precisely the following:
\begin{itemize}
     \item[\upshape(1)] $\nu^{-2s}\times\tau$, where $\tau\in\Irr(\G_2)$ is infinite-dimensional and $\Ha_{1,1}$-distinguished.

     \item[\upshape(2)] $\nu_{\mathds{1}_1}^{s+\frac{n-1}{2}}\times\nu^{-2s-1}$ if $s\notin\{-\frac n2,0\}$, and its unique irreducible subrepresentation otherwise.

     \item[\upshape(3)] $\nu_{\mathds{1}_1}^{s-\frac{n-1}{2}}\times\nu^{-2s+1}$ if $s\notin\{0,\frac n2\}$, and its unique irreducible subrepresentation otherwise.
\end{itemize}
\end{conj}
\begin{remark}

\label{rem:exceptional-symmetry}
At $s=0$, the unique irreducible subrepresentations occurring in
parts~\textup{(2)} and~\textup{(3)} are both isomorphic to $\mathds{1}_n$. Hence, the families in the statement of the
conjecture are not pairwise disjoint. At the remaining exceptional
values, the unique irreducible subrepresentation in
part~\textup{(2)} for
$
s=-\frac n2
$
is isomorphic to
$
\nu^{-2s}\widetilde{\mathcal Q}_n,
$
whereas the unique irreducible subrepresentation in
$
s=\frac n2
$
is isomorphic to
$
\nu^{-2s}\mathcal Q_n
$; see \S\ref{Qn} for notation. 

\end{remark}

       To provide evidence for our conjecture, we verify it for $n=3$ and $n=4$ in Sections~\ref{G_3} and \ref{G_4}, respectively. The study of $\GL_{n-1}(\F)$-distinction is closely connected with the study of $\GL_{1}(\F)\times \GL_{n-1}(\F)$-distinction. When $s=0$, our conjectural classification is analogous to the results of Prasad~\cite[Theorem~1]{Pra93} and Venketasubramanian~\cite[Theorem~1.1]{Ven13},
who classified $\GL_{n-1}(\F)$-distinguished representations of
$\GL_n(\F)$ in the split case. Moreover, in the same setting, our
classification is closely related to the classification of modular
representations of $\GL_n(\F)$ distinguished by $\GL_{n-1}(\F)$ obtained
by S\'echerre and Venketasubramanian~\cite{Venketmod}.
       The above conjecture will also be useful for studying the classification of representations of $\G_n$ admitting generalized linear periods with respect to $(\Ha_{p,q},\chi_s)$ for $p\geq 2$.

       As another consequence of Conjecture~\ref{conj}, we obtain the following result concerning the existence of generalized linear periods for generic representations of $\G_n$ with respect to $(\Ha_{1,n-1},\chi_s)$. Here, by a generic representation of $\G_n$, we mean an irreducible smooth representation $\pi$ such that its Jacquet--Langlands transfer $\JL(\pi)$ is a generic representation of $\GL_{2n}(\F)$ (see \S\ref{JLT} for more details).

\begin{corollary}\label{cor}
     A generic irreducible smooth representation of $\G_n$, with $n>3$, does not admit a generalized linear period with respect to $(\Ha_{1,n-1},\chi_s)$.
\end{corollary}

      For $s=0$, the above corollary is known by Anandavardhanan \emph{et al.}~\cite[Theorem~2.5]{anandavardhanan2024sign}. Now, assuming Conjecture~\ref{conj}, we establish the following characterization of the representations $\pi$ admitting generalized linear periods with respect to $(\Ha_{1,n-1},\chi_s)$ in terms of their Langlands parameters $\mathfrak{L}(\pi)$, illustrating how the existence of a generalized linear period for $\pi$ is reflected in the arithmetic invariants attached to it. Here, by $\mathfrak{L}(\pi)$, we mean $\mathfrak{L}(\JL(\pi))$, where $\JL(\pi)\in\Irr(\GL_{2n}(\F))$ is the Jacquet--Langlands transfer of $\pi$, as described in Section~\ref{JLTLP}.

\begin{theorem}\label{LP}
      Assume that Conjecture~\ref{conj} holds, and let $n>2$. An irreducible smooth representation $\pi$ of $\G_n$ admits a generalized linear period with respect to $(\Ha_{1,n-1},\chi_s)$ if and only if its Langlands parameter $\mathfrak{L}(\pi)$ admits a Weil--Deligne subrepresentation
isomorphic to $\mathfrak{L}(\nu^{-2s})$, the $(2n-4)$-dimensional
Langlands parameter of the character $\nu^{-2s}$ of $\G_{n-2}$,
      such that the four-dimensional quotient
$\mathfrak L(\pi)/\mathfrak L(\nu^{-2s})$
is the Langlands parameter of either the trivial representation $\mathds{1}_2$ of $\G_2$, in
which case
$s=\pm\frac{n-2}{2};$
or an irreducible infinite-dimensional
$\Ha_{1,1}$-distinguished representation of $\G_2$.
\end{theorem}

     From the proof of the above theorem, we obtain an important observation concerning the symplectic Langlands parameters. Anandavardhanan \emph{et al.}~\cite[Theorem~1.1]{anandavardhanan2024sign} proved that if a generic irreducible smooth representation of $\G_n$ admits a linear period, then its Langlands parameter is symplectic. In contrast, we observe that a non-generic representation admitting a generalized linear period with respect to $(\Ha_{1,n-1},\chi_s)$ need not have a symplectic Langlands parameter. For example, consider $\pi=\nu_{\mathds{1}_1}^{\,s+\frac{n-1}{2}}\times\nu^{-2s-1}$ with $s\notin\left\{-\frac n2,0\right\}$. By Conjecture~\ref{conj}, $\pi$ admits a generalized linear period with respect to $(\Ha_{1,n-1},\chi_s)$. However, its Langlands parameter  
     $$\mathfrak{L}(\pi)=\nu^{2s+n-1}\operatorname{Sp}_2\oplus\nu^{-2s+n-3}\operatorname{Sp}_2\oplus\cdots\oplus\nu^{-2s-n+1}\operatorname{Sp}_2$$ 
     is not symplectic.

     Finally, we study the $L$-packet associated with an irreducible representation of $\G_n$ admitting a linear period with respect to $\Ha_{1,n-1}$. Let $\theta:\G_n\to\G_n$ denote the involution defined in \eqref{theta}, which determines the symmetric pair $(\G_n,\Ha_{1,n-1})$. We prove the invariance of the $L$-packet predicted by the Lapid--Prasad conjecture, as recalled in $\S$\ref{LPC}; see \cite[Conjecture 2]{Prasad2015} and \cite[Conjecture 1.1]{Kapon2025}.
     More precisely, we establish the following theorem.

\begin{theorem}\label{Lapid-Prasad}
Assume that Conjecture~\ref{conj} holds for $s=0$. Let $n>2$ and
let $\pi\in\Irr(\G_n)$ be $\Ha_{1,n-1}$-distinguished. Then the
$L$-packet of $\pi$ is invariant under the involution
\[
\rho\longmapsto\widetilde{\rho}^{\,\theta},
\]
where $\widetilde{\rho}$ denotes the contragredient representation and $\widetilde{\rho}^{\,\theta}(g)=\widetilde{\rho}(\theta(g))$.
\end{theorem}

        Next, we outline the strategy for verifying Conjecture~\ref{conj} for $n\leq4$ and establishing Theorems~\ref{LP} and \ref{Lapid-Prasad}. It follows from Theorem~\ref{anand} that, for $n>2$, an irreducible supercuspidal representation of $\G_n$ does not admit generalized linear periods with respect to $(\Ha_{1,n-1},\chi_s)$. Therefore, it suffices to consider irreducible non-supercuspidal representations. Such representations arise as irreducible quotients of parabolically induced representations of the form $\pi_1\times\pi_2$, where $\pi_1\in\Irr(\G_k)$ and $\pi_2\in\Irr(\G_{n-k})$ for some $1\leq k\leq n-1$. Using Mackey theory for the restriction of parabolically induced representations, we first obtain a comprehensive list of representations that may admit generalized linear periods with respect to $(\Ha_{1,n-1},\chi_s)$. These representations need not be irreducible. We determine their Jordan--Hölder constituents using the Zelevinsky and Langlands classifications. Although permuting the inducing data preserves the set of Jordan--Hölder constituents, it may interchange the roles of subrepresentations and quotients; we therefore retain the order whenever this distinction is needed. Finally, we determine which of these irreducible constituents admit a generalized linear period with respect to $(\Ha_{1,n-1},\chi_s)$, thereby proving Conjecture~\ref{conj} for $n=3$ and $n=4$. To prove Theorem~\ref{LP}, we first compute the Jacquet--Langlands transfers $\JL(\pi)\in\Irr(\GL_{2n}(\F))$ of the irreducible representations $\pi\in\Irr(\G_n)$ occurring in Conjecture~\ref{conj}. We then determine the corresponding Langlands parameters $\mathfrak{L}(\JL(\pi))$, which are $2n$-dimensional representations of the Weil--Deligne group $W_{\F}'$. The resulting description of these Langlands parameters leads to the characterization stated in Theorem~\ref{LP}. Finally, to prove Theorem \ref{Lapid-Prasad}, we use the classification of irreducible representations admitting a linear period with respect to $\Ha_{1,n-1}$. 
        Under the local Langlands correspondence for the fixed inner form $\G_n$, every $L$-packet is a singleton. Thus, in the present setting, invariance of the $L$-packet is equivalent to $\pi\simeq\widetilde{\pi}^{\,\theta}$.

           This article is organized as follows. In Section~\ref{pre}, we fix the notation and recall basic facts on representations of $\G_n$, together with known results on generalized linear periods with respect to $(\Ha_{1,n-1},\chi_s)$. We also derive necessary conditions for the existence of generalized linear periods with respect to $(\Ha_{1,n-1},\chi_s)$ using the geometric lemma. In Section~\ref{G_3}, we prove the conjecture for $n=3$, while Section~\ref{G_4} is devoted to the proof for $n=4$. In Section~\ref{JLTLP}, we recall the Jacquet--Langlands transfer and the description of Langlands parameters, which we use to prove Theorem~\ref{LP}. We conclude Section~\ref{JLTLP} by proving the Theorem~\ref{Lapid-Prasad}.

\section{Notation and Preliminaries}\label{pre}

     \noindent Let \(\F\) be a non-Archimedean local field of characteristic zero, and let \(\D\) be a quaternion division algebra over \(\F\) with \(\dim_{\F} \D=4\). For \(n\in \mathbb{N}\), set \(\G_n=\GL_n(\D)\). Throughout, $\G_0$ denotes the trivial group. We denote by \(\Nrd_{\D/\F}:\G_n\to \F^\times\) the reduced norm map, and define \(\nu_n(g)=|\Nrd_{\D/\F}(g)|_{\F},\ g\in \G_n\). When the index is clear from the context, we simply write \(\nu\) instead of \(\nu_n\). Let $p$ and $q$ be non-negative integers with \(p+q=n\), and put \(\delta_{p,q}=\diag(I_p,-I_q)\). Consider the inner automorphism of $\G_n$ given by 
     \begin{equation}\label{theta}
\theta = \Ad(\delta_{p,q}) : g \longmapsto \delta_{p,q}\,g\,\delta_{p,q}^{-1}.
\end{equation}
     Its fixed-point subgroup is $\Ha_{p,q}=\left\{\diag(g_1,g_2):\;g_1\in \G_p,\;g_2\in \G_{q}\right\}.$ For \(s\in \mathbb{R}\), let $\alpha_s:\F^\times \rightarrow\mathbb{C}^\times$ denote the character defined by $\alpha_s(x)=|x|_{\F}^{\,2s},~x\in\F^\times$. Define a character $\chi_s$ of $\Ha_{p,q}$ by
     $$\chi_s(\diag(g_1,g_2))=\alpha_s\left(\Nrd_{\D/\F}(g_1)\Nrd_{\D/\F}(g_2)^{-1}\right)=\nu(g_1)^{2s}\nu(g_2)^{-2s}.$$

         Let $\Rep(\G_n)$ denote the category of smooth complex representations of $\G_n$ of finite length, $\Irr(\G_n)$ the set of equivalence classes of irreducible representations, and $\mathcal{C}(\G_n)$ the subset of irreducible supercuspidal representations. For $\pi\in\Rep(\G_n)$, let $\widetilde{\pi}$ denote its contragredient. Unless stated otherwise, all representations are smooth. 
         For a composition $\alpha=(n_1,\ldots,n_r)$ of $n$, let $\Par_\alpha$ denote the standard parabolic subgroup of $\G_n$ with Levi subgroup $\M_\alpha\simeq\G_{n_1}\times\cdots\times\G_{n_r}$. Given $\rho_i\in\Rep(\G_{n_i})$, we write $$\rho_1\times\cdots\times\rho_r=\Ind_{\Par_\alpha}^{\G_n}(\rho_1\otimes\cdots\otimes\rho_r)$$ for normalized parabolic induction. 
         For $\Psi=\rho_1\times\cdots\times\rho_r,
$
we set
$
\Psi^\vee
=
\widetilde{\rho_r}\times\cdots\times\widetilde{\rho_1}.
$
If $\M=\G_l\times\G_{n-l}$ is a maximal Levi subgroup of $\G_n$, we denote the corresponding normalized Jacquet module of $\pi$ by $r_{(l,n-l)}(\pi)$.

\subsection{Zelevinsky and Langlands classification}\label{ZLC}
        The results presented in this subsection follow \cite{tadic1990induced,ming}. Let $k$ be a positive integer and let $\rho\in\mathcal{C}(\G_k)$. There exists a unique positive integer $l_\rho$ such that the induced representation $\rho\times\nu^l\rho$ is reducible if and only if $l\in\{\pm l_\rho\}$. In particular, if $\rho\in\mathcal{C}(\G_1)$, then $l_\rho=2$ when $\dim(\rho)=1$ and $l_\rho=1$ otherwise. Moreover, if $\rho\in\mathcal{C}(\G_2)$, then $l_\rho=1$.
        Let us denote $\nu_{\rho}=\nu^{l_{\rho}}$.
        For integers $a\le b$, the finite ordered sequence of the form $$\Delta:=[a,b]_{(\rho)}=(\nu^{a}_{\rho}\rho,\nu^{a+1}_{\rho}\rho,\ldots,\nu^{b}_{\rho}\rho)$$ is called a segment associated with $\rho$. The length of segment $[a,b]_{(\rho)}$ is $b-a+1$. The contragredient of $\Delta$ is defined by $\widetilde{\Delta} := [-b, -a]_{(\widetilde{\rho})}$. Let $\Delta=[a,b]_{(\rho)}$ and $\Delta'=[a',b']_{(\rho')}$ be two segments. We say that $\Delta$ and $\Delta'$ are \emph{linked} if $\Delta\cup\Delta'$ is a segment, but neither $\Delta\subseteq\Delta'$ nor $\Delta'\subseteq\Delta$. If $\Delta$ and $\Delta'$ are linked and $\nu^{a'}_{\rho'}\rho'=\nu^{(a+j)}_{\rho}\rho$ for some $j>0$, then $\Delta$ is said to \emph{precede} $\Delta'$.

        Segments provide a systematic way to construct new representations of $\G_n$, forming the basis of the Zelevinsky and Langlands classifications. To each segment $\Delta = [a,b]_{(\rho)}$, we associate an irreducible representation $\mathcal{S}(\Delta)$ (resp. $\mathcal{Q}(\Delta)$), defined as the unique irreducible subrepresentation (resp. quotient) of $\nu_\rho^{a}\rho \times \nu_\rho^{a+1}\rho \times \cdots \times \nu_\rho^{b}\rho$. 
        For instance, if $\Delta=\big[-\frac{n-1}{2},\frac{n-1}{2}\big]_{(\mathds{1}_1)}$, we have $\mathcal{S}(\Delta)=\mathds{1}_n$, the trivial representation of $\G_n$, and $\mathcal{Q}(\Delta)=\St_n$, the Steinberg representation of $\G_n$. Similarly, if  $\Delta = \big[-\frac{n-1}{2}, \frac{n-1}{2}\big]_{(\sigma)}$ with $\sigma \in \Irr(\G_1)$ and $\dim(\sigma)>1$, then we obtain $\mathcal{S}(\Delta)=\Sp_n(\sigma)$, the Speh representation of $\G_n$, and $\mathcal{Q}(\Delta)=\St_n(\sigma)$, the generalized Steinberg representation of $\G_n$.

        Let $\Seg$ denote the set of all segments in $\mathcal{C}(\G)$, and let $\mathcal{M}(\Seg)$ be the set of all finite multiset of segments.
        Let $\mathfrak a=\{\Delta_1,\ldots,\Delta_r\}$ be a
multisegment. Choose an ordering of its segments such that
$\Delta_i$ does not precede $\Delta_j$ whenever $i<j$, and set
$$
\lambda(\mathfrak a)
=
\mathcal S(\Delta_1)\times\cdots\times\mathcal S(\Delta_r).
$$
Then $\lambda(\mathfrak a)$ has a unique irreducible
subrepresentation, denoted by $\mathcal S(\mathfrak a)$. Its
isomorphism class is independent of the chosen admissible ordering.
Similarly, set
$$
\pi(\mathfrak a)
=
\mathcal Q(\Delta_1)\times\cdots\times\mathcal Q(\Delta_r)
$$
using an admissible Langlands ordering. Its unique irreducible
quotient is denoted by $\mathcal Q(\mathfrak a)$.
        
        An elementary operation on $\mathfrak{a}\in \mathcal{M}(\Seg)$ consists of replacing a pair of linked segments $\{\Delta_1,\Delta_2\}$ in $\mathfrak{a}$ with the pair $\{\Delta_1\cup\Delta_2,\;\Delta_1\cap\Delta_2\}$, where the empty segment is omitted if
$\Delta_1\cap\Delta_2=\varnothing$. We define a partial order on $\mathcal{M}(\Seg)$ by declaring $\mathfrak{b} \le \mathfrak{a}$ if $\mathfrak{b}$ can be obtained from $\mathfrak{a}$ through a finite sequence of elementary operations. Now, we recall some results from \cite{ming} that are key tools for the explicit analysis of Jordan--Hölder constituents of representations of the form $\lambda(\mathfrak{a})$ of $\G_n$.

\begin{lemma}[{\cite[Lemma~5.12]{ming}, \cite[Proposition~4.3]{tadic1990induced}}]\label{l1}
        Let $\Delta_1$ and $\Delta_2$ be two segments. The representation $\pi=\mathcal{S}(\Delta_1)\times \mathcal{S}(\Delta_2)$ (respectively, $\pi=\mathcal{Q}(\Delta_1)\times \mathcal{Q}(\Delta_2)$) is irreducible if and only if $\Delta_1$ and $\Delta_2$ are not linked. If $\Delta_1$ and $\Delta_2$ are linked, let $\Delta_3=\Delta_1\cup\Delta_2$ and $\Delta_4=\Delta_1\cap\Delta_2$. Then $\pi$ has length $2$. If $\Delta_2\prec\Delta_1$, then $\pi$ has a unique irreducible subrepresentation $\mathcal{S}(\Delta_1,\Delta_2)$ (resp., $\mathcal{Q}(\Delta_3)\times \mathcal{Q}(\Delta_4)$) and a unique irreducible quotient $\mathcal{S}(\Delta_3)\times \mathcal{S}(\Delta_4)$ (resp., $\mathcal{Q}(\Delta_1,\Delta_2)$). If $\Delta_1\prec\Delta_2$, then the subrepresentation and quotient are interchanged.
\end{lemma}

\begin{lemma}[{\cite[Corollary~5.15]{ming}, \cite[Proposition~4.4]{tadic1990induced}}]\label{l2}
      For $\mathfrak{a},\mathfrak{b}\in\mathcal{M}(\Seg)$, the representation $\mathcal{S}(\mathfrak{b})$ (resp., $\mathcal{Q}(\mathfrak{b})$) occurs as a Jordan--Hölder constituent of $\lambda(\mathfrak{a})$ (resp., $\pi(\mathfrak{a})$) if and only if $\mathfrak{b}\leq\mathfrak{a}$.
\end{lemma}

\subsection{The representations $\mathcal{Q}_n$ and $\mathcal{S}_n$}\label{Qn}

In this subsection, we highlight two irreducible representations of $\G_n$, namely $\mathcal{Q}_n$ and $\mathcal{S}_n$, that play a fundamental role throughout the article. We first observe that $\nu$ is the unique irreducible quotient of the representation $\nu_{\mathds{1}_1}^{\frac{n-1}{2}}\times\nu_{\mathds{1}_1}^{\frac{n-3}{2}}\times\cdots\times\nu_{\mathds{1}_1}^{-\frac{n-3}{2}}\in\Rep(\G_{n-1})$. We begin with the following definition.

\begin{definition}\label{Ln}
For $n>2$, we define $\mathcal{Q}_n$ as the unique irreducible quotient of $\nu\times\nu_{\mathds{1}_{1}}^{\frac{n+1}{2}}$, and $\mathcal{S}_n$ as the unique irreducible subrepresentation of $\nu^{n}\St_2\times\mathds{1}_{n-2}$.

\end{definition}

By Lemma~\ref{l1}, $\mathcal{Q}_n$ is well defined and fits into the short exact sequence
\[
0\longrightarrow\nu^2\longrightarrow\nu\times\nu_{\mathds{1}_{1}}^{\frac{n+1}{2}}\longrightarrow \mathcal{Q}_n\longrightarrow0.
\]

The proof of the following result is analogous to that of Venketasubramanian~\cite[Lemma~3.4]{Ven13}.

\begin{lemma}\label{venk}
Let $n>2$ and $\chi$ be a character of $\G_2$. Then
$
\chi\St_2\times\mathds{1}_{n-2}
$
is reducible if and only if $\chi=\nu^{\pm n}$.
In particular, $\nu^n{\St_2}\times\mathds{1}_{n-2}$ has length two, with unique irreducible subrepresentation $\mathcal{S}_n$ and unique irreducible quotient $\mathcal{Q}_n$.
\end{lemma}

\begin{remark}\label{LnZn}

The representation $\mathcal{Q}_n$ and its contragredient admit descriptions in both the Zelevinsky and Langlands classifications:

\begin{itemize}

\item $\mathcal{Q}_n=\mathcal{S}\left([-\frac{n-3}{2},\frac{n-1}{2}]_{(\mathds{1}_1)},[\frac{n+1}{2}]_{(\mathds{1}_1)}\right)\simeq \mathcal{Q}\left([-\frac{n-3}{2}]_{(\mathds{1}_1)},\ldots,[\frac{n-3}{2}]_{(\mathds{1}_1)},[\frac{n-1}{2},\frac{n+1}{2}]_{(\mathds{1}_1)}\right)$.

\item $\widetilde{\mathcal{Q}}_n=\mathcal{S}\left([-\frac{n+1}{2}]_{(\mathds{1}_1)},[-\frac{n-1}{2},\frac{n-3}{2}]_{(\mathds{1}_1)}\right)\simeq \mathcal{Q}\left([-\frac{n-3}{2}]_{(\mathds{1}_1)},\ldots,[\frac{n-3}{2}]_{(\mathds{1}_1)},[-\frac{n+1}{2},-\frac{n-1}{2}]_{(\mathds{1}_1)}\right)$.

\end{itemize}

\end{remark}

\subsection{Some known results on generalized linear periods}  
       In this subsection, we recall several results concerning generalized linear periods. We begin with the following lemma of Mitra \emph{et al.}~\cite[Lemma~2.3]{mitra17}.

\begin{lemma}\label{basic}
       Let $p,q$ be nonnegative integers with $p+q=n$, and let $s\in\mathbb{R}$. If $\pi\in\Rep(\G_n)$ admits a generalized linear period with respect to $(\Ha_{p,q},\chi_s)$, then so does an Jordan--Hölder constituent of $\pi$.
\end{lemma}

The following result follows by an argument analogous to the proof of
\cite[Lemma~3.1 and Lemma~3.2]{yang2022linear}.

\begin{lemma}\label{cont}

\noindent{\upshape(1)} Let $\chi$ be a character of $\G_n$ admitting a generalized linear period with respect to $(\Ha_{1,n-1},\chi_s)$. Then $s=0$ and $\chi=\mathds{1}_n$.

\noindent{\upshape(2)}
Let $n_1,\ldots,n_t$ be positive integers such that
$n_1+\cdots+n_t=n$, and let
$\pi_i\in\Irr(\G_{n_i})$ for $1\leq i\leq t$. If
$
\Psi=\pi_1\times\cdots\times\pi_t
$
admits a generalized linear period with respect to
$(\Ha_{1,n-1},\chi_s)$, then
$
\Psi^\vee
=
\widetilde{\pi_t}\times\cdots\times\widetilde{\pi_1}$
admits a generalized linear period with respect to
$(\Ha_{1,n-1},\chi_{-s})$.

\end{lemma}

Next, we recall the results of Anandavardhanan \emph{et al.}~\cite[Theorem~3.8 and Theorem~C.3]{anandavardhanan2024sign} on generalized linear periods of discrete series representations and their products.
\begin{theorem}\label{anand}

\noindent{\upshape(1)} For $n>2$, no discrete series representation of $\G_n$ admits a generalized linear period with respect to $(\Ha_{1,n-1},\chi_s)$.

\noindent{\upshape(2)} Let $\pi_1,\pi_2$ be discrete series representations such that $\pi_1\times\pi_2\in\Irr(\G_n)$, where $n\ge4$. Then $\pi_1\times\pi_2$ does not admit a generalized linear period with respect to $(\Ha_{1,n-1},\chi_s)$.

\end{theorem}

\subsection{Consequence of Geometric Lemma}\label{orbit}

            Let $\G_n=\GL_n(\D)$. For $1\leq k\leq n-1$, let $\Par_{k,n-k}$ denote the standard parabolic subgroup of $\G_n$ corresponding to the partition $(k,n-k)$. In this subsection, we use the geometric lemma to derive necessary conditions for the existence of a generalized linear period with respect to $(\Ha_{1,n-1},\chi_s)$ for the representation $\Ind_{\Par_{k,n-k}}^{\G_n}(\pi)$. For this, we use the orbit decomposition of $\Par_{k,n-k}\backslash\G_n/\Ha_{1,n-1}$, which consists of three orbits; see~\cite[Section~4]{yang2022linear} and~\cite[\S3.2]{anandavardhanan2024sign} for details. The first two orbits are closed, while the third is open with stabilizer $L_{n,k}=\G_{k-1}\times\Delta(\G_1)\times\G_{n-k-1}$, where $\Delta(\G_1)=\{(g,g):g\in\G_1\}$.
Let $\omega_{s,n,k}$ denote the character of $L_{n,k}$ defined by
\[
\omega_{s,n,k}
=
\nu^{-2s+n-k-1}\otimes\mathds{1}_{\Delta(\G_1)}
\otimes\nu^{-(2s+k-1)}.
\]
Thus, for $(a,g,g,b)\in L_{n,k}$, we have
\[
\omega_{s,n,k}(a,g,g,b)
=
\nu(a)^{-2s+n-k-1}\nu(b)^{-(2s+k-1)}.
\]
The geometric lemma then yields the following necessary condition;
see~\cite[Corollary~5.2]{yang2022linear}.

\begin{lemma}\label{lem3}
          Let $\pi_1\in\Rep(\G_k)$ and $\pi_2\in\Rep(\G_{n-k})$, with $1\leq k\leq n-1$. If $\pi_1\times\pi_2$ admits a generalized linear period with respect to $(\Ha_{1,n-1},\chi_s)$, then at least one of the following holds:

         \begin{itemize}
             \item[\upshape(1)]  
             $\Hom_{\G_k} \left(\pi_1,\nu^{-2s+n-k-2}\right)\neq0$  and $\pi_2$ is $(\Ha_{1,n-k-1},\chi_{s+\frac{k}{2}})$-distinguished. \hfill\textnormal{(closed orbit I)}

          \item[\upshape(2)]  $\pi_1$ is $(\Ha_{1,k-1},\chi_{s+\frac{k-n}{2}})$-distinguished and $\Hom_{\G_{n-k}}(\pi_2,\nu^{-2s-k+2} )\neq 0$.  \hfill\textnormal{(closed orbit II)}

          \item[\upshape(3)]  $\Hom_{L_{n,k}}\left(r_{(k-1,1)}(\pi_1)\otimes r_{(1,n-k-1)}(\pi_2),\omega_{s,n,k}\right)\neq0. \hfill\textnormal{(open orbit III)}$ \end{itemize}
\end{lemma}

The conditions arising from the closed orbits $(\mathrm{I})$ and $(\mathrm{II})$ in Lemma~\ref{lem3} are also sufficient for the existence of generalized linear periods with respect to $(\Ha_{1,n-1},\chi_s)$; see~\cite[Proposition~3.2.16]{offb}. Thus, the two closed orbits give the following sufficient conditions.

\begin{lemma}\label{rem1}
For $1\leq k\leq n-1$, the following smooth representations of $\G_n$ admit generalized linear periods with respect to $(\Ha_{1,n-1},\chi_s)$:
\begin{enumerate}
\item[\upshape(1)] $\nu^{-2s+n-k-2}\times\eta$, where $\eta\in\Rep(\G_{n-k})$ is $(\Ha_{1,n-k-1},\chi_{s+\frac{k}{2}})$-distinguished.
\item[\upshape(2)] $\eta\times\nu^{-(2s+k-2)}$, where $\eta\in\Rep(\G_k)$ is $(\Ha_{1,k-1},\chi_{s+\frac{k-n}{2}})$-distinguished.
\end{enumerate}
\end{lemma}

The open orbit condition in $(3)$ of Lemma~\ref{lem3} gives the following necessary condition.

\begin{remark}
If the open orbit\/ $\mathrm{(III)}$ contributes to the
$(\Ha_{1,n-1},\chi_s)$-distinction of $\pi_1\times\pi_2$, then the
representation
$
r_{(k-1,1)}(\pi_1)\otimes r_{(1,n-k-1)}(\pi_2),
$
admits an Jordan--Hölder constituent of the form
$
\nu^{-2s+n-k-1}\otimes\rho\otimes\widetilde{\rho}
\otimes\nu^{-(2s+k-1)}
$
for some $\rho\in\Irr(\G_1)$.

\end{remark}

\begin{remark}
In proving Conjecture~\ref{conj} for $n=3$, we are required to take $k=1$, whereas for $n=4$ we take $k=2$. Thus, by Lemma~\ref{rem1}, the $\Ha_{1,1}$-distinction of irreducible representations of $\G_2$ plays a crucial role. We therefore recall the classification of irreducible representations of $\G_2$ that are $\Ha_{1,1}$-distinguished. First, every such representation is self-dual, that is, if $\pi$ is $\Ha_{1,1}$-distinguished, then $\pi\simeq\widetilde{\pi}$; see \cite[Theorem~3.2]{raghu}. The complete classification is due to Raghuram~\cite[Theorem~1.1]{raghuram2007restriction} and Gan--Takeda~\cite[Theorem~8.6]{gan}. More explicitly, the non-supercuspidal distinguished representations consist of $\mathds{1}_2$, the representations $\chi\,\St_2$ with $\chi|_{(\D^\ast)^2}=1$, the representations $\sigma\times\widetilde{\sigma}$ for irreducible representation $\sigma$ of $\D^\ast$, and the representations generalized Steinberg and Speh attached to a self-dual irreducible representation $\sigma$ of $\D^\ast$ with $\dim(\sigma)>1$. In the latter case, exactly one of $\St_2(\sigma)$ and $\Sp_2(\sigma)$ is $\Ha_{1,1}$-distinguished: $\St_2(\sigma)$ is distinguished when $\omega_\sigma\circ\Nrd_{\D/\F}$ is non-trivial, while $\Sp_2(\sigma)$ is distinguished when $\omega_\sigma\circ\Nrd_{\D/\F}$ is trivial; see \cite[Corollary~5]{DP}. On the other hand, a self-dual supercuspidal representation of $\G_2$ is $\Ha_{1,1}$-distinguished whenever its central character is trivial.
\end{remark}

\section{Verification of Conjecture~\ref{conj} for \texorpdfstring{$n=3$}{n=3}}\label{G_3}

In this section, we classify the irreducible smooth representations of
$\G_3$ admitting a generalized linear period with respect to
$(\Ha_{1,2},\chi_s)$ and thereby verify Conjecture~\ref{conj} for
$n=3$.
         
\begin{theorem}\label{G3}
 The irreducible smooth $\G_3$-representations admitting a generalized linear period with respect to $(\Ha_{1,2},\chi_s)$ are precisely the following:

\begin{enumerate}
       \item[\upshape(1)] 
       $\nu^{-2s}\times\tau$, where $\tau\in\Irr(\G_2)$ is infinite-dimensional and $\Ha_{1,1}$-distinguished.

      \item[\upshape(2)]  $\nu_{\mathds{1}_1}^{s+1}\times\nu^{-2s-1}$ if $s\notin\{-\frac32,0\}$, and its unique irreducible subrepresentation otherwise.

      \item[\upshape(3)] 
      $\nu_{\mathds{1}_1}^{s-1}\times\nu^{-2s+1}$ if $s\notin\{0,\frac32\}$, and its unique irreducible subrepresentation otherwise.
\end{enumerate}
\end{theorem}

\begin{proof}

\noindent
By Theorem~\ref{anand}, no irreducible supercuspidal representation of $\G_3$ admits a generalized linear period with respect to $(\Ha_{1,2},\chi_s)$. Hence, every $\pi\in\Irr(\G_3)$ admitting a generalized linear period with respect to $(\Ha_{1,2},\chi_s)$ is an irreducible quotient of either $\pi_1\times\pi_2$ or $\pi_2\times\pi_1$, where $\pi_i\in\Irr(\G_i)$. If $\pi$ is a quotient of $\pi_2\times\pi_1$, then $\widetilde{\pi}$ is a quotient of $\widetilde{\pi}_1\times\widetilde{\pi}_2$ by the compatibility of contragredient with normalized parabolic induction. Since $\chi_s^{-1}=\chi_{-s}$, $\widetilde{\pi}$ admits a generalized linear period with respect to $(\Ha_{1,2},\chi_{-s})$. Thus, after replacing $(\pi,s)$ by $(\widetilde{\pi},-s)$, it suffices to consider irreducible quotients of $\pi_1\times\pi_2$.
Let $\pi$ be such a quotient. A nonzero $(\Ha_{1,2},\chi_s)$-equivariant linear form on $\pi$ pulls back through the canonical surjection $\pi_1\times\pi_2\twoheadrightarrow\pi$ to a nonzero such form on $\pi_1\times\pi_2$. Hence, $\pi_1\times\pi_2$ admits a generalized linear period with respect to $(\Ha_{1,2},\chi_s)$. Applying Lemma~\ref{lem3} with $n=3$ and $k=1$, we obtain the following necessary conditions on $\pi_1\times\pi_2$ to admit a generalized linear period with respect to $(\Ha_{1,2},\chi_s)$:

  \begin{equation}\label{3.1} \left. 
 \begin{array}{ll} \mathrm{(I)} & \pi_1=\nu_{\mathds{1}_1}^{-s} \quad\text{and}\quad \pi_2\text{ is } (\Ha_{1,1},\chi_{s+\frac12})\text{-distinguished}. \\[2mm] \mathrm{(II)} & 
 \pi_1=\nu_{\mathds{1}_1}^{s-1} \quad\text{and}\quad  \pi_2=\nu^{-2s+1}. \\[2mm] \mathrm{(III)} &
 \Hom_{\Delta(\G_1)\times\G_1} \left( \pi_1\otimes r_{(1,1)}(\pi_2), \mathds{1}_{\Delta(\G_1)}\otimes\nu^{-2s} \right) \neq0. \end{array} \right\} \end{equation}
The conditions in \eqref{3.1} determine the possible inducing data
$(\pi_1,\pi_2)$. For the subsequent case-by-case analysis, we consider all 
Jordan--Hölder constituents of the resulting representations $\pi_1\times\pi_2$, rather
than only their irreducible quotients. This allows us to account for the
reversed ordering simultaneously without repeating the analysis. We divide the
analysis into four cases according to the form of $\pi_2\in\Irr(\G_2)$.

\noindent\textbf{Case (1).}
Suppose that $\pi_2$ is supercuspidal, a Speh representation, or a
generalized Steinberg representation. It follows from
\eqref{3.1} that $\pi$ is one of
\[
\nu_{\mathds{1}_1}^{-s}\times\tau,\qquad
\nu_{\mathds{1}_1}^{-s}\times\Sp_2(\sigma),\qquad
\nu_{\mathds{1}_1}^{-s}\times\St_2(\sigma'),
\]
where $\tau\in\mathcal{C}(\G_2)$ is self-dual with trivial central
character, while $\sigma,\sigma'\in\Irr(\G_1)$ are of dimension greater
than one and have trivial and nontrivial central characters,
respectively. By Lemma~\ref{l1}, each of these induced representations
are irreducible and admits a generalized linear period with respect to
$(\Ha_{1,2},\chi_s)$ arising from the closed orbit 
\eqref{3.1}\textup{(I)}.

\smallskip
\noindent\textbf{Case (2).}
Suppose that $\pi_2$ is a character. It follows from \eqref{3.1} that
the relevant induced representation is one of the following:
\[
\nu_{\mathds{1}_1}^{-s}\times\mathds{1}_2
\quad\text{with }s=-\frac12,\qquad
\nu_{\mathds{1}_1}^{\,s-1}\times\nu^{-(2s-1)},\qquad
\nu_{\mathds{1}_1}^{\,s+1}\times\nu^{-(2s+1)}.
\]

We consider these three possibilities separately.

\medskip

\noindent\textbf{(2a)}
Let $I=\nu_{\mathds{1}_1}^{-s}\times\mathds{1}_2$ with $s=-\frac12$. 
By Lemma~\ref{l1}, $I$ is irreducible and admits a generalized linear
period with respect to $(\Ha_{1,2},\chi_s)$ arising from the closed
orbit~\eqref{3.1}\textup{(I)}.

\medskip

\noindent\textbf{(2b)}
Let $I=\nu_{\mathds{1}_1}^{\,s-1}\times\nu^{-(2s-1)}$.
If $s\notin\{0,\frac32\}$, then $I$ is irreducible and admits a
generalized linear period with respect to $(\Ha_{1,2},\chi_s)$ arising
from the closed orbit~\eqref{3.1}\textup{(II)}. We next consider the cases $s\in\left\{0,\frac{3}{2}\right\}$.

\begin{enumerate}
\item[\upshape(i)]
Suppose that $s=0$. Then $I$ has length two, with unique irreducible
subrepresentation $\theta_1=\nu^{-2s}$
and unique irreducible quotient $\theta_2
=
\mathcal{S}
\bigl(
[-s-1]_{(\mathds{1}_1)},
[-s,-s+1]_{(\mathds{1}_1)}
\bigr)$. 
Since $\theta_1$ is the trivial representation, it admits a generalized linear period with respect to $(\Ha_{1,2},\chi_s)$ by Lemma~\ref{cont}. The
M\oe glin--Waldspurger algorithm\cite[\S 1]{BR07} gives
$$\theta_2
\simeq
\mathcal{Q}
\bigl(
[-s+1]_{(\mathds{1}_1)},
[-s-1,-s]_{(\mathds{1}_1)}
\bigr),$$
which is the unique irreducible quotient of
$\nu_{\mathds{1}_1}^{-s+1}
\times
\mathcal{Q}
\bigl([-s-1,-s]_{(\mathds{1}_1)}\bigr)$.
The latter induced representation does not occur among the
possibilities in \eqref{3.1} and therefore does not admit a generalized
linear period with respect to $(\Ha_{1,2},\chi_s)$. Consequently,
neither does $\theta_2$.

\item[\upshape(ii)]
Suppose that $s=\frac32$. Then $I$ has length two, with unique
irreducible quotient
$\theta_1=\nu^{-2s+2}$
and unique irreducible subrepresentation
$\theta_2 =
\mathcal{S}
\bigl(
[-s+2]_{(\mathds{1}_1)},
[-s,-s+1]_{(\mathds{1}_1)}
\bigr)$.
By Lemma~\ref{cont}, $\theta_1$ does not admit a generalized linear
period with respect to $(\Ha_{1,2},\chi_s)$. On the other hand, $I$
admits such a period arising from the closed orbit~\eqref{3.1}\textup{(II)}.
It therefore follows from Lemma~\ref{basic} that $\theta_2$ admits a
generalized linear period with respect to $(\Ha_{1,2},\chi_s)$.
\end{enumerate}

\medskip

\noindent\textbf{(2c)}
Let $I=\nu_{\mathds{1}_1}^{\,s+1}\times\nu^{-(2s+1)}$.
Suppose first that $s\notin\{-\frac32,0\}$. By Lemma~\ref{l1}, $I$ is
irreducible. If, in addition, $s\notin\{-1,-\frac12\}$, then $I$ is
the unique irreducible quotient of
$ I'=
\nu_{\mathds{1}_1}^{-s}
\times
(\nu_{\mathds{1}_1}^{\,s+1}
\times
\nu_{\mathds{1}_1}^{-s-1}).
$
By Lemma~\ref{l1}, $I'$ has length two, with $I$ and
$\nu_{\mathds{1}_1}^{\,s+1}
\times
\nu^{-2s-1}\St_2$
as its Jordan--Hölder constituents. The representation
$I'$ admits a generalized linear period with respect to
$(\Ha_{1,2},\chi_s)$ arising from the closed orbit~\eqref{3.1}\textup{(I)}, whereas
$\nu_{\mathds{1}_1}^{\,s+1}
\times
\nu^{-2s-1}\St_2$ does not admit such a period by
\eqref{3.1}. Lemma~\ref{basic} therefore implies that $I$ admits a
generalized linear period with respect to $(\Ha_{1,2},\chi_s)$.
It remains to examine the values
$s\in\left\{-\frac32,-1,-\frac12,0\right\}$.

\begin{enumerate}
\item[\upshape(i)]
Suppose that $s=-\frac32$. Then $I$ has length two, with unique
irreducible quotient
$\theta_1=\nu^{-2s-2}$
and unique irreducible subrepresentation
$\theta_2
=
\mathcal{S}
\bigl(
[-s-2]_{(\mathds{1}_1)},
[-s-1,-s]_{(\mathds{1}_1)}
\bigr)$.
By Lemma~\ref{cont}, $\theta_1$ does not admit a generalized linear
period with respect to $(\Ha_{1,2},\chi_s)$. Moreover, by
M\oe glin--Waldspurger algorithm\cite[\S 1]{BR07}, we have
$$\theta_2
\simeq
\mathcal{Q}
\bigl(
[-s]_{(\mathds{1}_1)},
[-s-2,-s-1]_{(\mathds{1}_1)}
\bigr).$$
By Lemma~\ref{l1}, the representation
$\nu_{\mathds{1}_1}^{-s}
\times
\mathcal{Q}
\bigl([-s-2,-s-1]_{(\mathds{1}_1)}\bigr)$
has  Jordan--Hölder constituents
$\theta_2$ and 
$\mathcal{Q}\bigl([-s-2,-s]_{(\mathds{1}_1)}\bigr)$.
By Theorem~\ref{anand}, the representation $\mathcal{Q}\bigl([-s-2,-s]_{(\mathds{1}_1)}\bigr)$ does not admit a generalized linear
period with respect to $(\Ha_{1,2},\chi_s)$, whereas the induced
representation above admits such a period arising from the closed
orbit~\eqref{3.1}\textup{(I)}. Hence,Lemma~\ref{basic} implies that $\theta_2$
admits a generalized linear period with respect to
$(\Ha_{1,2},\chi_s)$.

\item[\upshape(ii)]
Suppose that $s=-1$. Then $I$ is irreducible by Lemma~\ref{l1}, and
$\widetilde{I}
\simeq
\nu_{\mathds{1}_1}^{-s-1}\times\nu^{2s+1}$.
It follows from \eqref{3.1} that $\widetilde{I}$ admits a generalized
linear period with respect to $(\Ha_{1,2},\chi_{-s})$ arising from the
closed orbit~\eqref{3.1}\textup{(II)}. Consequently, by Lemma \ref{cont}, $I$ admits a generalized
linear period with respect to $(\Ha_{1,2},\chi_s)$.

\item[\upshape(iii)]
Suppose that $s=-\frac12$. By Lemma~\ref{l1}, $I$ is irreducible and
admits a generalized linear period with respect to
$(\Ha_{1,2},\chi_s)$ arising from the closed orbit~\eqref{3.1}\textup{(I)}.

\item[\upshape(iv)]
Suppose that $s=0$. By Lemma~\ref{l1}, $I$ has length two, with unique
irreducible subrepresentation
$\theta_1=\nu^{-2s}$
and unique irreducible quotient
$\theta_2
=
\mathcal{S}
\bigl(
[-s+1]_{(\mathds{1}_1)},
[-s-1,-s]_{(\mathds{1}_1)}
\bigr)$.
Since $\theta_1$ is the trivial representation, it admits a generalized
linear period with respect to $(\Ha_{1,2},\chi_s)$. By
Case~\textup{(2b)}, the representation $\widetilde{\theta}_2$ does not admit a generalized
linear period with respect to $(\Ha_{1,2},\chi_{-s})$. Hence,by Lemma \ref{cont},
$\theta_2$ also does not admit a generalized linear period with respect to
$(\Ha_{1,2},\chi_s)$.
\end{enumerate}

\smallskip
\noindent\textbf{Case (3).}
Suppose that $\pi_2$ is a twist of the Steinberg representation. It
follows from \eqref{3.1} that the relevant induced representations are
$$\nu_{\mathds{1}_1}^{-s}\times\chi\St_2
~\text{with}~ \chi^2=\mathds{1}_1,
~\text{or}~ \nu_{\mathds{1}_1}^{s-1}
\times\nu^{-2s+1}\St_2.$$
 We consider these two possibilities separately.

\medskip

\noindent\textbf{(3a)}
Let
$I=\nu_{\mathds{1}_1}^{-s}\times\chi\St_2$,
with $\chi^2=\mathds{1}_1$.
If $\chi\neq\mathds{1}_1$, or if $\chi=\mathds{1}_1$ and
$s\neq\pm\frac32$, then Lemma~\ref{l1} shows that $I$ is irreducible
and admits a generalized linear period with respect to $(\Ha_{1,2},\chi_s)$ arising from the closed orbit~\eqref{3.1}\textup{(I)}.
Suppose now that $\chi=\mathds{1}_1$ and $s=\frac32$. Then $I$ has the following
Jordan--Hölder constituents:$$\nu^{-2s}\mathcal{S}_3~\text{and}~\nu^{-2s}\mathcal{Q}_3.$$
By Theorem~\ref{anand}, the representation
$\nu^{-2s}\mathcal{S}_3$ does not admit a generalized linear period
with respect to $(\Ha_{1,2},\chi_s)$, while the question of whether the representation $\nu^{-2s}\mathcal{Q}_3$ admit a generalized linear period with respect to $(\Ha_{1,2},\chi_s)$ has already been discussed in Case~\textup{(2)}.
The remaining value $\chi=\mathds{1}_1$ and $s=-\frac32$ yields no new
Jordan--Hölder constituents, since it is recovered from the preceding case by
taking the contragredient and replacing $s$ by $-s$.

\medskip

\noindent\textbf{(3b)}
Let $I=\nu_{\mathds{1}_1}^{s-1}\times\nu^{-2s+1}\St_2$.
If $s=\frac12$, then $I$ is irreducible and admits a generalized linear
period with respect to $(\Ha_{1,2},\chi_s)$ arising from the closed
orbit~\eqref{3.1}\textup{(I)}. Assume henceforth that $s\neq\frac12$.
If $s\notin\{0,\frac32\}$, then $I$ is irreducible by
Lemma~\ref{l1}, and
$\widetilde{I}
\simeq
\nu_{\mathds{1}_1}^{-s+1}\times\nu^{2s-1}\St_2$.
It follows from \eqref{3.1} that $\widetilde{I}$ does not admit a
generalized linear period with respect to
$(\Ha_{1,2},\chi_{-s})$. Consequently, $I$ does not admit a generalized
linear period with respect to $(\Ha_{1,2},\chi_s)$ by Lemma \ref{cont}. It remains to
consider $s=0$ and $s=\frac32$.

\begin{enumerate}
\item[\upshape(i)]
Suppose that $s=0$. Then $I$ has the following Jordan--Hölder constituents:
$$
\theta_1=\nu^{-2s-2}\mathcal{Q}_3
~\text{and}~
\theta_2=\nu^{-2s-2}\mathcal{S}_3.
$$
The question of whether $\theta_1$ admits a generalized linear period with respect to $(\Ha_{1,2},\chi_s)$ has already been addressed in Case~\textup{(2c)}, whereas Theorem~\ref{anand} shows that $\theta_2$
does not admit a generalized linear period with respect to
$(\Ha_{1,2},\chi_s)$.

\item[\upshape(ii)]
Suppose that $s=\frac32$. Then $I$ has the following Jordan--Hölder constituents:
$\theta_1 =
\mathcal{Q}\bigl([-s,-s+2]_{(\mathds{1}_1)}\bigr)$
and
$\theta_2 =
\mathcal{Q}
\bigl(
[-s+2]_{(\mathds{1}_1)},
[-s,-s+1]_{(\mathds{1}_1)}
\bigr)$.
By Theorem~\ref{anand}, $\theta_1$ does not admit a generalized linear
period with respect to $(\Ha_{1,2},\chi_s)$. Moreover, by M\oe glin--Waldspurger algorithm \cite[\S 1]{BR07}, we have
$$
\theta_2
\simeq
\mathcal{S}
\bigl(
[-s]_{(\mathds{1}_1)},
[-s+1,-s+2]_{(\mathds{1}_1)}
\bigr),
$$
and this representation is the unique irreducible quotient of
$\nu_{\mathds{1}_1}^{-s}
\times
\mathcal{S}\bigl([-s+1,-s+2]_{(\mathds{1}_1)}\bigr)$.
Since the latter induced representation does not admit a generalized linear period with respect to $(\Ha_{1,2},\chi_s)$, Lemma~\ref{basic} implies that neither does its irreducible quotient $\theta_2$.
\end{enumerate}
    
\smallskip
\noindent\textbf{Case~(4).}
Suppose that $\pi_2$ is an irreducible principal series representation.
It follows from \eqref{3.1} that the relevant induced representations are
of the form:
\[
\begin{aligned}
&\nu_{\mathds{1}_1}^{-s}\times\sigma\times\widetilde{\sigma}
\qquad\text{where }\sigma\in\Irr(\G_1)\text{ with }
\sigma\not\simeq\nu_{\sigma}^{\pm1}\widetilde{\sigma},\\
&\sigma\times\nu_{\mathds{1}_1}^{-s}\times\widetilde{\sigma}
\quad\text{or}\quad
\sigma\times\widetilde{\sigma}\times\nu_{\mathds{1}_1}^{-s}
\qquad\text{where }\sigma\in\Irr(\G_1)\text{ with }
\sigma\not\simeq\nu_{\mathds{1}_1}^{s\pm1}.
\end{aligned}
\]
according to the closed orbit(I) and open orbit(III) respectively. Without loss of generality, it is sufficient to study the Jordan--Hölder constituents of $I=\nu_{\mathds{1}_1}^{-s}\times\sigma\times\widetilde{\sigma}$ with $\sigma\in\Irr(\G_1)$. If
$\sigma\notin
\left\{\nu_{\mathds{1}_1}^{-s\pm1},
\nu_{\mathds{1}_1}^{ s\pm1},
\,\nu_{\sigma}^{\pm 1}\widetilde{\sigma}
\right\}$,
then $I$ is irreducible and admits a generalized
linear period with respect to $(\Ha_{1,2},\chi_s)$ arising from the
closed orbit~\eqref{3.1}\textup{(I)}. We now examine the remaining possibilities.

\medskip

\noindent\textbf{(4a)}
Suppose that
$\sigma\simeq\nu_{\sigma}^{\pm 1}\widetilde{\sigma}$ with $\dim(\sigma)>1$.
Then $I$ has the following Jordan--Hölder constituents:
$\theta_1=\nu_{\mathds{1}_1}^{-s}\times\Sp_2(\sigma)$ and
$\theta_2=\nu_{\mathds{1}_1}^{-s}\times\St_2(\sigma)$. 
It follows from \eqref{3.1} that $\theta_1$ admits a generalized linear
period with respect to $(\Ha_{1,2},\chi_s)$ when the central character
$\omega_\sigma$ is trivial, whereas $\theta_2$ admits such a period
when $\omega_\sigma$ is nontrivial.

\medskip

\noindent\textbf{(4b)}
Suppose that
$\sigma=\nu_{\mathds{1}_1}^{-s+1}$ and 
$s\notin\left\{0,\frac12,1,\frac32\right\}$.
Then
$I=
\nu_{\mathds{1}_1}^{-s}
\times
\nu_{\mathds{1}_1}^{-s+1}
\times
\nu_{\mathds{1}_1}^{s-1}$
has the Jordan--Hölder constituents
$
\theta_1
=
\nu_{\mathds{1}_1}^{s-1}\times\nu^{-(2s-1)}
~\text{and}~
\theta_2
=
\nu_{\mathds{1}_1}^{s-1}
\times
\nu^{-2s+1}\St_2.
$
The question of whether the representations  $\theta_1$ and $\theta_2$ admit a generalized linear period with respect to \((\Ha_{1,2},\chi_s)\)  has already been
determined in Cases~\textup{(2b)} and~\textup{(3b)}, respectively. It
remains to examine
$s\in\left\{0,\frac12,1,\frac32\right\}$.

\begin{enumerate}
\item[\upshape(i)]
Suppose that $s=0$. Then $I$ has the following Jordan--Hölder constituents:
\[
\begin{aligned}
\theta_1&=
\mathcal{S}\bigl(
[-s]_{(\mathds{1}_1)},
[-s+1]_{(\mathds{1}_1)},
[-s-1]_{(\mathds{1}_1)}
\bigr),\\
\theta_2&=
\mathcal{S}\bigl(
[-s-1,-s]_{(\mathds{1}_1)},
[-s+1]_{(\mathds{1}_1)}
\bigr),\\
\theta_3&=
\mathcal{S}\bigl(
[-s-1]_{(\mathds{1}_1)},
[-s,-s+1]_{(\mathds{1}_1)}
\bigr),\\
\theta_4&=
\mathcal{S}\bigl([-s-1,-s+1]_{(\mathds{1}_1)}\bigr).
\end{aligned}
\]
   By Theorem~\ref{anand}, $\theta_1$ does not admit a generalized linear
period with respect to $(\Ha_{1,2},\chi_s)$, while the  question of whether the representations $\theta_2$, $\theta_3$, and $\theta_4$ admit a generalized linear period with respect to \((\Ha_{1,2},\chi_s)\) has already been determined in
Cases~\textup{(2b)} and~\textup{(2c)}.

\item[\upshape(ii)]
Suppose that $s=\frac12$. Then $I$ has the Jordan--Hölder constituents
$
\theta_1
=
\nu_{\mathds{1}_1}^{-s}\times\nu^{-2s+1}
~\text{and}~
\theta_2
=
\nu_{\mathds{1}_1}^{-s}\times\nu^{-2s+1}\St_2.
$
The question of whether they admit a generalized linear period with respect to $(\Ha_{1,2},\chi_s)$ has already been addressed in Cases~\textup{(2b)} and~\textup{(3b)}, respectively.

\item[\upshape(iii)]
Suppose that $s=1$. Then $I$ has the Jordan--Hölder constituents
$
\theta_1
=
\nu_{\mathds{1}_1}^{-s+1}\times\nu^{-2s+1}
~\text{and}~
\theta_2
=
\nu_{\mathds{1}_1}^{-s+1}\times\nu^{-2s+1}\St_2.
$
By Case~\textup{(2b)}, $\theta_1$ admits a generalized linear
period with respect to $(\Ha_{1,2},\chi_s)$, whereas
Case~\textup{(3b)} shows that $\theta_2$ does not.

\item[\upshape(iv)]
Suppose that $s=\frac32$. Then $I$ has the following Jordan--Hölder constituents:
\[
\begin{aligned}
\theta_1&=
\mathcal{S}\bigl(
[-s]_{(\mathds{1}_1)},
[-s+1]_{(\mathds{1}_1)},
[-s+2]_{(\mathds{1}_1)}
\bigr),\\
\theta_2&=
\mathcal{S}\bigl(
[-s,-s+1]_{(\mathds{1}_1)},
[-s+2]_{(\mathds{1}_1)}
\bigr),\\
\theta_3&=
\mathcal{S}\bigl(
[-s]_{(\mathds{1}_1)},
[-s+1,-s+2]_{(\mathds{1}_1)}
\bigr),\\
\theta_4&=
\mathcal{S}\bigl([-s,-s+2]_{(\mathds{1}_1)}\bigr).
\end{aligned}
\]
By Theorem~\ref{anand}, the representation $\theta_1$ does not admit a generalized
linear period with respect to $(\Ha_{1,2},\chi_s)$. The representation $\theta_2$ admit a generalized linear period with respect to $(\Ha_{1,2},\chi_s)$ by
Case~\textup{(2b)(ii)}, whereas Case~\textup{(3b)(ii)} shows that
$\theta_3$ does not admit a generalized
linear period with respect to $(\Ha_{1,2},\chi_s)$. Finally, $\theta_4$ is a nontrivial character of $\G_3$. Hence, it does not admit a generalized linear period with
respect to $(\Ha_{1,2},\chi_s)$ by Lemma~\ref{cont}.

\end{enumerate}

\medskip

\noindent\textbf{(4c)} Suppose that $\sigma\in
\left\{\nu_{\mathds{1}_1}^{-s-1},
\nu_{\mathds{1}_1}^{ s\pm1},
\,\nu_{\sigma}^{\pm 1}\widetilde{\sigma}~|~\Dim(\sigma)=1
\right\}.$
These possibilities
are treated in the same manner and do not yield additional representations
admitting a generalized linear period with respect to
$(\Ha_{1,2},\chi_s)$, up to taking the contragredient and replacing $s$
by $-s$. 

Combining Cases~\textup{(1)}--\textup{(4)}, we obtain exactly
the three families stated in the theorem. This proves the result and,
in particular, verifies Conjecture~\ref{conj} for $n=3$.
\end{proof}

\section{Proof of Conjecture~\ref{conj} for \texorpdfstring{$n=4$}{n=4}}\label{G_4}

In this section, we classify the irreducible smooth representations of $\G_4$
admitting a generalized linear period with respect to $(\Ha_{1,3},\chi_s)$.
This establishes Conjecture~\ref{conj} for $n=4$.

\begin{theorem}\label{G4}
The irreducible smooth $\G_4$-representations admitting a generalized linear period with respect to $(\Ha_{1,3},\chi_s)$ are precisely the following:
\begin{enumerate}
\item[\upshape(1)]
$\nu^{-2s}\times\tau$, where $\tau\in\Irr(\G_2)$ is infinite-dimensional and
$\Ha_{1,1}$-distinguished;

\item[\upshape(2)]
$\nu_{\mathds{1}_1}^{s+\frac32}\times\nu^{-2s-1}$ if
$s\notin\{-2,0\}$, and its unique irreducible subrepresentation otherwise;

\item[\upshape(3)]
$\nu_{\mathds{1}_1}^{s-\frac32}\times\nu^{-2s+1}$ if
$s\notin\{0,2\}$, and its unique irreducible subrepresentation otherwise.
\end{enumerate}
\end{theorem}

\begin{proof}

The proof follows the strategy used in Theorem~\ref{G3}. By
Theorem~\ref{anand}, no irreducible supercuspidal representation of
$\G_4$ admits a generalized linear period with respect to
$(\Ha_{1,3},\chi_s)$. Hence, every $\theta\in\Irr(\G_4)$ admitting a
generalized linear period with respect to $(\Ha_{1,3},\chi_s)$ is a
non-supercuspidal representation and occurs as an irreducible quotient
of
\[
\pi=\rho\times\tau,\qquad
\rho\in\Irr(\G_k),\quad
\tau\in\Irr(\G_{4-k}),\qquad 1\leq k\leq3.
\]
If $\theta$ admits a generalized linear period with respect to
$(\Ha_{1,3},\chi_s)$, then a nonzero
$(\Ha_{1,3},\chi_s)$-equivariant linear form on $\theta$ pulls back
through the canonical surjection $\pi\twoheadrightarrow\theta$ to a
nonzero such form on $\pi$. Thus, $\pi$ itself admits a generalized
linear period with respect to $(\Ha_{1,3},\chi_s)$. By an argument
similar to that used in \cite[Lemma~17]{hariom}, we may restrict to the
case $\rho,\tau\in\Irr(\G_2)$. Applying Lemma~\ref{lem3} with $n=4$ and
$k=2$, we obtain the following necessary conditions on
$\rho\times\tau$ to admit a generalized linear period with respect to
$(\Ha_{1,3},\chi_s)$:

\begin{equation}\label{4.1}
\left.
\begin{array}{ll}
(\mathrm{I})
& \rho=\nu^{-2s}~ \text{and }\tau\text{ is }
  (\Ha_{1,1},\chi_{s+1})\text{-distinguished}.\\[2mm]
(\mathrm{II})
& \rho\text{ is }(\Ha_{1,1},\chi_{s-1})\text{-distinguished and }
   \tau=\nu^{-2s}.\\[2mm]
(\mathrm{III})
& \Hom_{\G_1\times\Delta(\G_1)\times\G_1}
  \!\left(
  r_{(1,1)}(\rho)\otimes r_{(1,1)}(\tau),
  \nu_{\mathds{1}_1}^{-(s-\frac12)}
  \otimes\mathds{1}_{\Delta(\G_1)}
  \otimes\nu_{\mathds{1}_1}^{-(s+\frac12)}
  \right)\neq0.
\end{array}
\right\}
\end{equation}

The conditions in \eqref{4.1} determine the possible inducing data
$(\rho,\tau)$. We therefore analyze $\rho\times\tau$ case by case.
As in the case $n=3$, we consider all Jordan--Hölder
constituents, which also accounts for the reversed ordering of the
inducing data. Thus, we analyze the constituents that can admit
generalized linear periods with respect to $(\Ha_{1,3},\chi_s)$.

\medskip
\noindent\textbf{Case~(1).}
Suppose that one of $\rho$ and $\tau$ is supercuspidal, generalized Steinberg,
or Speh, whereas the other is an infinite-dimensional representation. Then $\pi$ satisfies none
of the conditions in \eqref{4.1}. Hence, it does not admit a generalized linear
period with respect to $(\Ha_{1,3},\chi_s)$.

\medskip
\noindent\textbf{Case~(2).}
Suppose that exactly one of $\rho$ and $\tau$ is a character and that the other
is supercuspidal, generalized Steinberg, or Speh. We may assume that $\rho$ is a
character. It follows from \eqref{4.1} that, up to permutation, $\pi$ is one of
\[
\nu^{-2s}\times\tau,
\qquad
\nu^{-2s}\times\St_2(\sigma),
\qquad
\nu^{-2s}\times\Sp_2(\sigma'),
\]
where $\tau\in\mathcal{C}(\G_2)$ has trivial central character and
$\sigma,\sigma'\in\Irr(\G_1)$ satisfy
$\dim(\sigma),\dim(\sigma')>1$, with $\sigma$ having trivial central character
and $\sigma'$ having nontrivial central character. By Lemma~\ref{l1}, each of
these induced representations is irreducible, and Lemma~\ref{rem1} shows that
each admits a generalized linear period with respect to
$(\Ha_{1,3},\chi_s)$ arising from the closed orbit \eqref{4.1}(I).

\medskip
\noindent\textbf{Case~(3).}
Suppose that both $\rho$ and $\tau$ are characters. It follows from \eqref{4.1} that $\pi$ must be one of the following:
$$
\nu^{-2s}\times\mathds{1}_2 \quad\text{with } s=-1,
\qquad
\mathds{1}_2\times\nu^{-2s} \quad\text{with } s=1,
\qquad
\nu^{-2s+2}\times\nu^{-2s-2} \quad\text{with } s=0.
$$
We analyze these cases separately.

\smallskip
\noindent\textbf{(3a)}
Let $\pi=\nu^{-2s}\times\mathds{1}_2$ with $s=-1$. Then $\pi$ has the following Jordan--Hölder constituents:
\[
\begin{aligned}
\theta_1
&=
\nu_{\mathds{1}_1}^{-s-\scriptstyle\frac12}
\times
\mathcal{S}\left(
\left[
\textstyle-s-\scriptstyle\frac32,~
\textstyle-s+\scriptstyle\frac12
\right]_{(\mathds{1}_1)}
\right),\\
\theta_2
&=
\mathcal{S}\left(
\left[
\textstyle-s-\scriptstyle\frac12,~
\textstyle-s+\scriptstyle\frac12
\right]_{(\mathds{1}_1)},
\left[
\textstyle-s-\scriptstyle\frac32,~
\textstyle-s-\scriptstyle\frac12
\right]_{(\mathds{1}_1)}
\right).
\end{aligned}
\]

By M\oe glin--Waldspurger algorithm\cite[\S 1]{BR07}, we have 
\[
\theta_2\simeq
\mathcal{Q}\left(
\left[
\textstyle-s-\scriptstyle\frac12,~
\textstyle-s+\scriptstyle\frac12
\right]_{(\mathds{1}_1)},
\left[
\textstyle-s-\scriptstyle\frac32,~
\textstyle-s-\scriptstyle\frac12
\right]_{(\mathds{1}_1)}
\right),
\]
which is the unique irreducible quotient of
\[
\mathcal{Q}\left(
\left[
\textstyle-s-\scriptstyle\frac12,~
\textstyle-s+\scriptstyle\frac12
\right]_{(\mathds{1}_1)}
\right)
\times
\mathcal{Q}\left(
\left[
\textstyle-s-\scriptstyle\frac32,~
\textstyle-s-\scriptstyle\frac12
\right]_{(\mathds{1}_1)}
\right).
\]
It is easy to observe from \eqref{4.1} that the latter induced representation does not admit a generalized linear period with respect to $(\Ha_{1,3},\chi_s)$; hence, neither does $\theta_2$. Since $\pi$ admits such a period, Lemma~\ref{basic} implies that $\theta_1$ admits a generalized linear period with respect to $(\Ha_{1,3},\chi_s)$.

\smallskip
\noindent\textbf{(3b)}
Let $\pi=\mathds{1}_2\times\nu^{-2s}$ with $s=1$. The same argument, with the
two factors interchanged, shows that the unique irreducible quotient of $\pi$
admits a generalized linear period with respect to $(\Ha_{1,3},\chi_s)$.

\smallskip
\noindent\textbf{(3c)}
Let $\pi=\nu^{-2s+2}\times\nu^{-2s-2}$ with $s=0$. Its Jordan--Hölder constituents are:

\[
\begin{aligned}
\theta_1
&=
\mathcal{S}\left(
\left[
\textstyle -s-\scriptstyle\frac32,~
\textstyle -s+\scriptstyle\frac32
\right]_{(\mathds{1}_1)}
\right),\\
\theta_2
&=
\mathcal{S}\left(
\left[
\textstyle -s+\scriptstyle\frac12,~
\textstyle -s+\scriptstyle\frac32
\right]_{(\mathds{1}_1)},
\left[
\textstyle -s-\scriptstyle\frac32,~
\textstyle -s-\scriptstyle\frac12
\right]_{(\mathds{1}_1)}
\right).
\end{aligned}
\]

By Lemma~\ref{cont}, $\theta_1$ admits a generalized linear period with respect to $(\Ha_{1,3},\chi_s)$. Furthermore, by M\oe glin--Waldspurger algorithm\cite[\S 1]{BR07}, we have
\[
\theta_2\simeq
\mathcal{Q}\left(
\left[
\textstyle-s+\scriptstyle\frac32
\right]_{(\mathds{1}_1)},
\left[
\textstyle-s-\scriptstyle\frac12,~
\textstyle-s+\scriptstyle\frac12
\right]_{(\mathds{1}_1)},
\left[
\textstyle-s-\scriptstyle\frac32
\right]_{(\mathds{1}_1)}
\right).
\]
This is the unique irreducible quotient of
$$
\nu_{\mathds{1}_1}^{-s+\scriptstyle\frac32}
\times
\mathcal{Q}\left(
\left[
\textstyle-s-\scriptstyle\frac12,~
\textstyle-s+\scriptstyle\frac12
\right]_{(\mathds{1}_1)},
\left[
\textstyle-s-\scriptstyle\frac32
\right]_{(\mathds{1}_1)}
\right),
$$
which does not admit a generalized linear period with respect to $(\Ha_{1,3},\chi_s)$. Hence, neither does $\theta_2$.

\medskip
\noindent\textbf{Case~(4).}
Suppose that both $\rho$ and $\tau$ are twists of the Steinberg representation. Then the possible representation $\pi$ according to $\eqref{4.1}$ is:
\[
\pi=\nu^{-2s}\St_2\times\nu^{-2s}\St_2~\text{with}~ s=0.
\]
Since $\pi$ is an irreducible product of discrete series representations,
Theorem~\ref{anand} shows that it does not admit a generalized linear period with respect to $(\Ha_{1,3},\chi_s)$.

\medskip
\noindent\textbf{Case~(5).}
Suppose that exactly one of $\rho$ and $\tau$ is a twist of the Steinberg
representation and the other is a principal-series representation. By
\eqref{4.1}, the possible induced representations are:
\[
\begin{aligned}
&\nu_{\mathds{1}_1}^{-s+\frac12}
 \times\nu_{\mathds{1}_1}^{s-\frac12}
 \times\nu^{-2s}\St_2,
&&s\notin\{0,1\},\\
&\nu^{s-\frac12}\times\nu^{-s+\frac12}\times\nu^{-2s}\St_2,
&&s\notin\{0,1\},\\
&\nu^{-2s}\St_2
 \times\nu_{\mathds{1}_1}^{-s-\frac12}
 \times\nu_{\mathds{1}_1}^{s+\frac12},
&&s\notin\{-1,0\},\\
&\nu^{-2s}\St_2
 \times\nu_{\mathds{1}_1}^{s+\frac12}
 \times\nu_{\mathds{1}_1}^{-s-\frac12},
&&s\notin\{-1,0\}.
\end{aligned}
\]
Consider the first representation. By Lemma~\ref{l1}, it is irreducible precisely when $s\neq-\frac12$. For $s\neq\frac12$, it satisfies none of the conditions in Lemma~\ref{lem3} with $n=4$ and $k=1$, and hence does not admit a generalized linear period with respect to $(\Ha_{1,3},\chi_s)$. When $s=\frac12$, it satisfies the condition associated with the open orbit~\eqref{4.1}\textup{(III)}, whereas its contragredient satisfies none of the conditions in Lemma~\ref{lem3}. Thus, its contragredient does not admit a generalized linear period with respect to $(\Ha_{1,3},\chi_{-s})$. By Lemma~\ref{cont}, the representation $\pi$ therefore does not admit a generalized linear period with respect to $(\Ha_{1,3},\chi_s)$.

It remains to consider $s=-\frac12$. In this case, $\pi$ has the Jordan--Hölder constituents,
\[
\nu_{\mathds{1}_1}^{-s+\frac12}\times\nu^{-2s-3}\mathcal{Q}_3
\qquad\text{and}\qquad
\nu_{\mathds{1}_1}^{-s+\frac12}\times\nu^{-2s-3}\mathcal{S}_3.
\]
Neither satisfies the conditions of Lemma~\ref{lem3} with $n=4$ and $k=1$. Thus, neither admits a generalized linear period with respect to $(\Ha_{1,3},\chi_s)$. The remaining three induced representations are treated similarly and yield no Jordan--Hölder constituents admitting a generalized linear period with respect to $(\Ha_{1,3},\chi_s)$.

\medskip
\noindent\textbf{Case~(6).}
Suppose that exactly one of $\rho$ and $\tau$ is a character and that the other
is a twist of Steinberg representation. It follows from \eqref{4.1} that $\pi$ must be one of the following:
\[
\begin{aligned}
&\nu^{-2s}\times\chi\St_2\qquad\text{or}\qquad
\chi\St_2\times\nu^{-2s}
\quad\text{with }\chi^2=\mathds{1}_1,\\
&\nu^{-2s}\St_2\times\nu^{-2s-2}
\quad\text{with }s=-1,\\
&\nu^{-2s+2}\times\nu^{-2s}\St_2
\quad\text{with }s=1.
\end{aligned}
\]
We analyze these cases one by one.

\noindent\textbf{(6a)}
Let $\pi=\chi\St_2\times\nu^{-2s}$, where $\chi^2=\mathds{1}_1$. If
    $\chi\neq\mathds{1}_1$, or if $\chi=\mathds{1}_1$ and $s\neq\pm2$, then $\pi$ is irreducible and admits a generalized linear period with respect to $(\Ha_{1,3},\chi_s)$ arising from the closed orbit~\eqref{4.1}\textup{(II)}. Suppose that $\chi=\mathds{1}_1$ and $s=2$. Then $\pi$ has the Jordan--Hölder constituents,
\[
\nu^{-2s}\mathcal{S}_4
\qquad\text{and}\qquad
\nu^{-2s}\mathcal{Q}_4.
\]
The representation $\nu^{2s}\widetilde{\mathcal{S}}_4$ is the unique
irreducible quotient of
$
\nu_{\mathds{1}_1}^{s+\frac12}\times\nu^{2s-3}\St_3.
$
The latter induced representation satisfies none of the conditions in Lemma~\ref{lem3} with $n=4$ and $k=1$, and hence does not admit a generalized linear period with respect to $(\Ha_{1,3},\chi_{-s})$. Consequently, $\nu^{2s}\widetilde{\mathcal{S}}_4$ does not admit a generalized linear period either, and Lemma~\ref{cont} implies that $\nu^{-2s}\mathcal{S}_4$ does not admit a generalized linear period with respect to $(\Ha_{1,3},\chi_s)$. Since $\pi$ admits a generalized linear period with respect to $(\Ha_{1,3},\chi_s)$, Lemma~\ref{basic} shows that $\nu^{-2s}\mathcal{Q}_4$ admits a generalized linear period with respect to $(\Ha_{1,3},\chi_s)$. 

Similarly, when $s=-2$, the only Jordan--Hölder constituent admitting a generalized linear period with respect to $(\Ha_{1,3},\chi_s)$ is $\nu^{-2s}\widetilde{\mathcal{Q}}_4$. When $\chi$ is a quadratic character of $\G_1$ and
$\pi=\nu^{-2s}\times\chi\St_2$, no new representation admitting a
generalized linear period with respect to $(\Ha_{1,3},\chi_s)$ occurs.

\noindent\textbf{(6b)}
Let $\pi=\nu^{-2s}\St_2\times\nu^{-2s-2}$ with $s=-1$. By Lemma~\ref{l1},
$\pi$ is irreducible and is the unique irreducible quotient of
\[
\pi'
=
\nu_{\mathds{1}_1}^{-s-\frac12}
\times
\left(
\nu_{\mathds{1}_1}^{-s+\frac12}\times\nu^{-2s-2}
\right).
\]
The representation $\pi'$ satisfies none of the conditions in
Lemma~\ref{lem3} with $n=4$ and $k=1$, and therefore it does not admit a
generalized linear period with respect to $(\Ha_{1,3},\chi_s)$. Hence, neither does its quotient $\pi$. Applying
Lemma~\ref{cont} gives the same conclusion for
$\nu^{-2s+2}\times\nu^{-2s}\St_2$ when $s=1$.

\medskip
\noindent\textbf{Case~(7).}
Suppose that exactly one of $\rho$ and $\tau$ is a character and that the other
is an irreducible principal-series representation. The possibilities arising
from \eqref{4.1} are
\[
\begin{aligned}
&\nu^{-2s}\times\sigma\times\widetilde{\sigma},
\quad \sigma\times\widetilde{\sigma}\times\nu^{-2s}
\quad\text{with }\sigma\not\simeq\nu_{\sigma'}^{\pm\frac12}\sigma'
\text{ for any self-dual }\sigma'\in\Irr(\G_1),\\
&\left(\nu_{\mathds{1}_1}^{-s+\frac12}\times\nu_{\mathds{1}_1}^{s+\frac32}\right)\times\nu^{-2s+2}
\quad\text{with }s\notin\{-1,0\},\\
&\left(\nu_{\mathds{1}_1}^{s+\frac32}\times\nu_{\mathds{1}_1}^{-s+\frac12}\right)\times\nu^{-2s+2}
\quad\text{with }s\notin\{-1,0\},\\
&\nu^{-2s+2}\times\left(\nu_{\mathds{1}_1}^{s-\frac32}\times\nu_{\mathds{1}_1}^{-s-\frac12}\right)
\quad\text{with }s\notin\{0,1\},\\
&\nu^{-2s+2}\times\left(\nu_{\mathds{1}_1}^{-s-\frac12}\times\nu_{\mathds{1}_1}^{s-\frac32}\right)
\quad\text{with }s\notin\{0,1\}.
\end{aligned}
\]
We analyze these cases separately.

\noindent\textbf{(7a)}
Let $\pi=\sigma\times\widetilde{\sigma}\times\nu^{-2s}$. If
$
\sigma\notin
\left\{
\nu_{\mathds{1}_1}^{-s\pm\frac32},
\nu_{\mathds{1}_1}^{s\pm\frac32}
\right\},
$
then $\pi$ is irreducible by Lemma~\ref{l1} and admits a generalized linear
period with respect to $(\Ha_{1,3},\chi_s)$ arising from the closed orbit~\eqref{4.1}\textup{(II)}. Now consider the remaining cases.

\begin{enumerate}
\item[\upshape(i)]
Suppose that $\sigma=\nu_{\mathds{1}_1}^{-s-\frac32}$ with
$s\notin\{-2,-1\}$. If $s\neq0$, then $\pi$ has the following Jordan--Hölder constituents:
\[
\begin{aligned}
\theta_1
&=
\nu_{\mathds{1}_1}^{s+\scriptstyle\frac32}
\times
\mathcal{S}\left(
\left[
\textstyle-s-\scriptstyle\frac32
\right]_{(\mathds{1}_1)},
\left[
\textstyle-s-\scriptstyle\frac12,~
\textstyle-s+\scriptstyle\frac12
\right]_{(\mathds{1}_1)}
\right),\\
\theta_2
&=
\mathcal{S}\left(
\left[
\textstyle-s-\scriptstyle\frac32,~
\textstyle-s+\scriptstyle\frac12
\right]_{(\mathds{1}_1)}
\right)
\times
\nu_{\mathds{1}_1}^{s+\scriptstyle\frac32}.
\end{aligned}
\]
Lemma~\ref{lem3}, applied with $n=4$ and $k=3$, shows that $\theta_2$ admits a generalized linear period with respect to $(\Ha_{1,3},\chi_s)$ arising from the closed orbit, whereas $\theta_1$ does not satisfy the conditions of the lemma and hence does not admit a generalized linear period with respect to $(\Ha_{1,3},\chi_s)$.
If $s=0$, then $\pi$ has the following Jordan--Hölder constituents:
\[
\begin{aligned}
\theta_1
&=
\mathcal{S}\left(
\left[
\textstyle-s-\scriptstyle\frac32,~
\textstyle-s+\scriptstyle\frac32
\right]_{(\mathds{1}_1)}
\right),\\
\theta_2
&=
\mathcal{S}\left(
\left[
\textstyle-s+\scriptstyle\frac32
\right]_{(\mathds{1}_1)},
\left[
\textstyle-s-\scriptstyle\frac32,~
\textstyle-s+\scriptstyle\frac12
\right]_{(\mathds{1}_1)}
\right),\\
\theta_3
&=
\mathcal{S}\left(
\left[
\textstyle-s-\scriptstyle\frac32
\right]_{(\mathds{1}_1)},
\left[
\textstyle-s-\scriptstyle\frac12,~
\textstyle-s+\scriptstyle\frac32
\right]_{(\mathds{1}_1)}
\right),\\
\theta_4
&=
\mathcal{S}\left(
\left[
\textstyle-s-\scriptstyle\frac32
\right]_{(\mathds{1}_1)},
\left[
\textstyle-s+\scriptstyle\frac32
\right]_{(\mathds{1}_1)},
\left[
\textstyle-s-\scriptstyle\frac12,~
\textstyle-s+\scriptstyle\frac12
\right]_{(\mathds{1}_1)}
\right).
\end{aligned}
\]
Here $\theta_1$ is the trivial representation and therefore admits a generalized linear period with respect to $(\Ha_{1,3},\chi_s)$.
Furthermore, by M\oe glin--Waldspurger algorithm\cite[\S 1]{BR07}, we have
\[
\theta_2\simeq
\mathcal{Q}\left(
\left[
\textstyle-s+\scriptstyle\frac12,~
\textstyle-s+\scriptstyle\frac32
\right]_{(\mathds{1}_1)},
\left[
\textstyle-s-\scriptstyle\frac12
\right]_{(\mathds{1}_1)},
\left[
\textstyle-s-\scriptstyle\frac32
\right]_{(\mathds{1}_1)}
\right),
\]
which is the unique irreducible quotient of
\[
\mathcal{Q}\left(
\left[
\textstyle-s+\scriptstyle\frac12,~
\textstyle-s+\scriptstyle\frac32
\right]_{(\mathds{1}_1)}
\right)
\times
\mathcal{Q}\left(
\left[
\textstyle-s-\scriptstyle\frac12
\right]_{(\mathds{1}_1)},
\left[
\textstyle-s-\scriptstyle\frac32
\right]_{(\mathds{1}_1)}
\right).
\]
The latter representation does not satisfy the conditions in \eqref{4.1} and hence does not admit a generalized linear period with respect to $(\Ha_{1,3},\chi_s)$. Therefore, neither does $\theta_2$. Taking the contragredient, we obtain the same conclusion for $\theta_3$. Finally, by M\oe glin--Waldspurger algorithm\cite[\S 1]{BR07}, we have
\[
\theta_4\simeq
\mathcal{Q}\left(
\left[
\textstyle-s+\scriptstyle\frac12,~
\textstyle-s+\scriptstyle\frac32
\right]_{(\mathds{1}_1)},
\left[
\textstyle-s-\scriptstyle\frac32,~
\textstyle-s-\scriptstyle\frac12
\right]_{(\mathds{1}_1)}
\right),
\]
which is the unique irreducible quotient of \[
\mathcal{Q}\left(
\left[
\textstyle-s+\scriptstyle\frac12,~
\textstyle-s+\scriptstyle\frac32
\right]_{(\mathds{1}_1)}\right) \times\left(
\left[
\textstyle-s-\scriptstyle\frac32,~
\textstyle-s-\scriptstyle\frac12
\right]_{(\mathds{1}_1)}
\right),
\]
The latter representation does not satisfy the conditions in
\eqref{4.1} and hence does not admit a generalized linear period with
respect to $(\Ha_{1,3},\chi_s)$.

\item[\upshape(ii)]
Suppose that $\sigma=\nu_{\mathds{1}_1}^{-s+\frac32}$ with
$s\notin\{1,2\}$. If $s\neq 0$, then $\pi$ has the following  Jordan--Hölder constituents:
\[
\begin{aligned}
\theta_1
&=
\nu_{\mathds{1}_1}^{s-\scriptstyle\frac32}
\times
\mathcal{S}\left(
\left[
\textstyle-s+\scriptstyle\frac32
\right]_{(\mathds{1}_1)},
\left[
\textstyle-s-\scriptstyle\frac12,~
\textstyle-s+\scriptstyle\frac12
\right]_{(\mathds{1}_1)}
\right),\\
\theta_2
&=
\nu_{\mathds{1}_1}^{s-\scriptstyle\frac32}
\times
\mathcal{S}\left(
\left[
\textstyle-s-\scriptstyle\frac12,~
\textstyle-s+\scriptstyle\frac32
\right]_{(\mathds{1}_1)}
\right).
\end{aligned}
\]
The representation $\widetilde{\theta}_1$ does not admit a generalized linear period with respect to $(\Ha_{1,3},\chi_{-s})$ by part~\textup{(i)}. Lemma~\ref{cont} therefore implies that $\theta_1$ does not admit a generalized linear period with respect to $(\Ha_{1,3},\chi_s)$.  On the other hand, Lemma~\ref{lem3} with $n=4$ and $k=1$ shows that $\theta_2$ admits a
generalized linear period with respect to $(\Ha_{1,3},\chi_s)$ arising from the closed orbit.
\end{enumerate}

The cases $\sigma=\nu_{\mathds{1}_1}^{s+\frac32}$ with
$s\notin\{-2,-1\}$ and $\sigma=\nu_{\mathds{1}_1}^{s-\frac32}$ with
$s\notin\{1,2\}$ are analogous and yield no additional Jordan--Hölder constituents admitting a generalized linear period  with respect to $(\Ha_{1,3},\chi_s)$, up to taking
the contragredient and replacing $s$ by $-s$.

\smallskip
\noindent\textbf{(7b)}
Let
$
\pi
=
\left(
\nu_{\mathds{1}_1}^{-s+\frac12}
\times\nu_{\mathds{1}_1}^{s+\frac32}
\right)
\times\nu^{-2s+2}
~\text{with}~ s\notin\{-1,0\}.
$
By Lemma~\ref{l1}, $\pi$ is irreducible if and only if $s\neq\frac12$. In
that case, writing
\[
\pi
=
\nu_{\mathds{1}_1}^{-s+\frac12}
\times
\left(
\nu_{\mathds{1}_1}^{s+\frac32}\times\nu^{-2s+2}
\right)
\]
shows that it satisfies none of the conditions in Lemma~\ref{lem3} with $n=4$
and $k=1$. Thus, $\pi$ does not admit a generalized linear period with respect to $(\Ha_{1,3},\chi_s)$. If $s=\frac12$, then $\pi$ has the following Jordan--Hölder constituents:  $$\nu_{\mathds{1}_1}^{-s+\frac12}\times\nu^{-2s+3}~ \text{and}~ \nu_{\mathds{1}_1}^{-s+\frac12}\times\nu^{-2s+1}\mathcal{Q}_3.$$ Neither satisfies the conditions of Lemma~\ref{lem3} with $n=4$ and $k=1$. The remaining three forms listed at the beginning of Case~\textup{(7)} are handled similarly and yield no additional irreducible representations admitting a generalized linear period with respect to $(\Ha_{1,3},\chi_s)$.

\medskip
\noindent\textbf{Case~(8).}
Suppose that both $\rho$ and $\tau$ are irreducible principal-series
representations. In this case, $\pi$ is one of the following:
\[
\begin{aligned}
&\nu_{\mathds{1}_1}^{-s+\frac12}
 \times\sigma\times\widetilde{\sigma}
 \times\nu_{\mathds{1}_1}^{-s-\frac12},\\
&\sigma\times\nu_{\mathds{1}_1}^{-s+\frac12}
 \times\widetilde{\sigma}\times\nu_{\mathds{1}_1}^{-s-\frac12},\\
&\nu_{\mathds{1}_1}^{-s+\frac12}
 \times\sigma\times\nu_{\mathds{1}_1}^{-s-\frac12}
 \times\widetilde{\sigma},\\
&\sigma\times\nu_{\mathds{1}_1}^{-s+\frac12}
 \times\nu_{\mathds{1}_1}^{-s-\frac12}
 \times\widetilde{\sigma},
\end{aligned}
\]
where
\[
\sigma\notin
\left\{
\nu_{\mathds{1}_1}^{\pm s+\frac32},
\nu_{\mathds{1}_1}^{\pm s-\frac12}
\right\}.
\]
For the above representations, the inducing data are the same up to
permutation. Hence, it suffices to analyze one of them. We therefore analyze only
$$\pi
=
\nu_{\mathds{1}_1}^{-s+\frac12}
\times\sigma\times\widetilde{\sigma}
\times\nu_{\mathds{1}_1}^{-s-\frac12}, \qquad \sigma\notin
\left\{
\nu_{\mathds{1}_1}^{\pm s+\frac32},
\nu_{\mathds{1}_1}^{\pm s-\frac12}
\right\}.$$
This representation is not irreducible, since the first and fourth
components are linked. Next, we distinguish cases according to the representation
$\sigma$ and analyze the resulting Jordan--Hölder constituents.

 \noindent \textbf{(8a)} Suppose that 
$
\sigma\notin
\left\{
\nu_{\mathds{1}_1}^{\pm s-\frac32},
\nu_{\mathds{1}_1}^{\pm s+\frac12},
\nu_{\sigma'}^{\pm\frac12}\sigma'
\;\middle|\;
\sigma'\in\Irr(\G_1)\text{ is self-dual}
\right\}$. 
Then $\pi$ has the Jordan--Hölder constituents
\[
\theta_1=\nu^{-2s}\times\sigma\times\widetilde{\sigma}
\qquad\text{and}\qquad
\theta_2=\nu^{-2s}\St_2\times\sigma\times\widetilde{\sigma}.
\]

 The representation $\theta_1$ admits a generalized linear period with respect to $(\Ha_{1,3},\chi_s)$ arising from the closed orbit \eqref{4.1}\textup{(I)}. The representation $\theta_2$ satisfies none of the conditions in \eqref{4.1} and therefore does not admit a generalized linear period with respect to $(\Ha_{1,3},\chi_s)$. It remains to consider the exceptional choices of $\sigma$.

\smallskip
\noindent\textbf{(8b)}
Suppose that $\sigma=\nu_{\mathds{1}_1}^{-s-\frac32}$ with
$s\notin\{-\frac32,-\frac12\}$. If $s\notin\{-2,-1,0\}$, then $\pi$ has the following
Jordan--Hölder constituents:
\[
\begin{aligned}
\theta_1
&=
\mathcal{S}\left(
\left[
\textstyle-s-\scriptstyle\frac32,~
\textstyle-s+\scriptstyle\frac12
\right]_{(\mathds{1}_1)}
\right)
\times
\nu_{\mathds{1}_1}^{s+\scriptstyle\frac32},\\
\theta_2
&=
\mathcal{S}\left(
\left[
\textstyle-s-\scriptstyle\frac32
\right]_{(\mathds{1}_1)},
\left[
\textstyle-s-\scriptstyle\frac12,~
\textstyle-s+\scriptstyle\frac12
\right]_{(\mathds{1}_1)}
\right)
\times
\nu_{\mathds{1}_1}^{s+\scriptstyle\frac32},\\
\theta_3
&=
\mathcal{S}\left(
\left[
-s+\scriptstyle\frac12
\right]_{(\mathds{1}_1)},
\left[
\textstyle-s-\scriptstyle\frac32,~
\textstyle-s-\scriptstyle\frac12
\right]_{(\mathds{1}_1)}
\right)
\times
\nu_{\mathds{1}_1}^{s+\scriptstyle\frac32},\\
\theta_4
&=
\mathcal{S}\left(
\left[
\textstyle-s+\scriptstyle\frac12
\right]_{(\mathds{1}_1)},
\left[
\textstyle-s-\scriptstyle\frac32
\right]_{(\mathds{1}_1)},
\left[
\textstyle-s-\scriptstyle\frac12
\right]_{(\mathds{1}_1)}
\right)
\times
\nu_{\mathds{1}_1}^{s+\scriptstyle\frac32}.
\end{aligned}
\]
The representations $\theta_1$ and $\theta_2$ were treated in Case~\textup{(7a)}. By \eqref{4.1}, $\theta_3$ does not admit a generalized linear period with respect to $(\Ha_{1,3},\chi_s)$. By the M\oe glin--Waldspurger algorithm~\cite[\S 1]{BR07}, we have $$
\theta_4\simeq
\mathcal{Q}\left(
\left[
\textstyle-s-\scriptstyle\frac32,~
\textstyle-s+\scriptstyle\frac12
\right]_{(\mathds{1}_1)}
\right)
\times
\nu_{\mathds{1}_1}^{s+\scriptstyle\frac32}.
$$ Another application of Lemma \ref{lem3} with $n=4$ and $k=1$ shows that $\theta_4$ does not admit a generalized linear period with respect to $(\Ha_{1,3},\chi_s)$.
We now consider the exceptional values $s=-2,-1,0$.
\begin{enumerate}
\item[\upshape(i)]
Suppose that $s=-2$. Then
$
\pi
=
\nu_{\mathds{1}_1}^{-s+\frac12}
\times\nu_{\mathds{1}_1}^{-s-\frac32}
\times\nu_{\mathds{1}_1}^{-s-\frac52}
\times\nu_{\mathds{1}_1}^{-s-\frac12}
$
has the following Jordan--Hölder constituents:
\[
\begin{aligned}
\theta_1
&=
\mathcal{S}\left(
\left[
\textstyle-s-\scriptstyle\frac52,~
\textstyle-s+\scriptstyle\frac12
\right]_{(\mathds{1}_1)}
\right),\\
\theta_2
&=
\mathcal{S}\left(
\left[
\textstyle-s+\scriptstyle\frac12
\right]_{(\mathds{1}_1)},
\left[
\textstyle-s-\scriptstyle\frac52,~
\textstyle-s-\scriptstyle\frac12
\right]_{(\mathds{1}_1)}
\right),\\
\theta_3
&=
\mathcal{S}\left(
\left[
\textstyle-s-\scriptstyle\frac32,~
\textstyle-s+\scriptstyle\frac12
\right]_{(\mathds{1}_1)},
\left[
\textstyle-s-\scriptstyle\frac52
\right]_{(\mathds{1}_1)}
\right),\\
\theta_4
&=
\mathcal{S}\left(
\left[
\textstyle-s-\scriptstyle\frac52,~
\textstyle-s-\scriptstyle\frac32
\right]_{(\mathds{1}_1)},
\left[
\textstyle-s-\scriptstyle\frac12,~
\textstyle-s+\scriptstyle\frac12
\right]_{(\mathds{1}_1)}
\right),\\
\theta_5
&=
\mathcal{S}\left(
\left[
\textstyle-s+\scriptstyle\frac12
\right]_{(\mathds{1}_1)},
\left[
\textstyle-s-\scriptstyle\frac32,~
\textstyle-s-\scriptstyle\frac12
\right]_{(\mathds{1}_1)},
\left[
\textstyle-s-\scriptstyle\frac52
\right]_{(\mathds{1}_1)}
\right),\\
\theta_6
&=
\mathcal{S}\left(
\left[
\textstyle-s-\scriptstyle\frac32
\right]_{(\mathds{1}_1)},
\left[
\textstyle-s-\scriptstyle\frac52
\right]_{(\mathds{1}_1)},
\left[
\textstyle-s-\scriptstyle\frac12,~
\textstyle-s+\scriptstyle\frac12
\right]_{(\mathds{1}_1)}
\right),\\
\theta_7
&=
\mathcal{S}\left(
\left[
-s+\scriptstyle\frac12
\right]_{(\mathds{1}_1)},
\left[
\textstyle-s-\scriptstyle\frac52,~
\textstyle-s-\scriptstyle\frac32
\right]_{(\mathds{1}_1)},
\left[
\textstyle-s-\scriptstyle\frac12
\right]_{(\mathds{1}_1)}
\right),\\
\theta_8
&=
\mathcal{S}\left(
\left[
\textstyle-s+\scriptstyle\frac12
\right]_{(\mathds{1}_1)},
\left[
\textstyle-s-\scriptstyle\frac32
\right]_{(\mathds{1}_1)},
\left[
\textstyle-s-\scriptstyle\frac52
\right]_{(\mathds{1}_1)},
\left[
\textstyle-s-\scriptstyle\frac12
\right]_{(\mathds{1}_1)}
\right).
\end{aligned}
\]
 The representation $\theta_1$ is a nontrivial character and hence does not admit a generalized linear period with respect to $(\Ha_{1,3},\chi_s)$ by Lemma~\ref{cont}. Next, $\theta_2$ is the unique irreducible quotient of $$
\mathcal{S}\left(
\left[
\textstyle-s-\scriptstyle\frac52,~
\textstyle-s-\scriptstyle\frac12
\right]_{(\mathds{1}_1)}
\right)
\times
\nu_{\mathds{1}_1}^{-s+\scriptstyle\frac12}.
$$
 By Lemma~\ref{lem3} with $n=4$ and $k=3$, the latter representation does not admit a generalized linear period  with respect to $(\Ha_{1,3},\chi_s)$. Consequently, neither does $\theta_2$.  By M\oe glin--Waldspurger algorithm~\cite[\S 1]{BR07}, $\theta_3$ is isomorphic to $\nu^{-2s}\widetilde{\mathcal{Q}_4}$, whose generalized linear period with respect to $(\Ha_{1,3},\chi_s)$ was already determined in Case (6). By Lemma~\ref{l1}, the representation $\theta_4$ is the unique irreducible quotient of $\nu^{-2(s+2)}\times\nu^{-2s}$. Since the latter representation does not admit a generalized linear period with respect to $(\Ha_{1,3},\chi_s)$, neither does $\theta_4$. Next, $\theta_5$ is isomorphic to$$
\mathcal{Q}\left(
\left[
\textstyle-s-\scriptstyle\frac12,~
\textstyle-s+\scriptstyle\frac12
\right]_{(\mathds{1}_1)},
\left[
\textstyle-s-\scriptstyle\frac52,~
\textstyle-s-\scriptstyle\frac32
\right]_{(\mathds{1}_1)}
\right),
$$
 which is the unique irreducible quotient of $\nu^{-2s}\St_2\times\nu^{-2(s+2)}\St_2$. The latter representation does not satisfy the conditions in
\eqref{4.1} and hence does not admit a generalized linear period with
respect to $(\Ha_{1,3},\chi_s)$.
 Hence, neither does $\theta_5$. Similarly, by M\oe glin--Waldspurger algorithm~\cite[\S 1]{BR07}, $\theta_6$ is isomorphic to $$
\mathcal{Q}\left(
\left[
\textstyle-s+\scriptstyle\frac12
\right]_{(\mathds{1}_1)},
\left[
\textstyle-s-\scriptstyle\frac52,~
\textstyle-s-\scriptstyle\frac12
\right]_{(\mathds{1}_1)}
\right),
$$ which is the unique irreducible quotient of $\nu^{-s+\frac12}_{\mathds{1}_1}\times \mathcal{Q}([{-s-\frac52},{-s-\frac12}]_{(\mathds{1}_1)})$. By Lemma~\ref{lem3} with $n=4$ and $k=1$, the latter representation does not admit a generalized linear period  with respect to $(\Ha_{1,3},\chi_s)$. Hence, neither does $\theta_6$. Furthermore, by M\oe glin--Waldspurger algorithm~\cite[\S 1]{BR07}, $\theta_7$ is isomorphic to $$
\mathcal{Q}\left(
\left[
\textstyle-s-\scriptstyle\frac12,~
\textstyle-s+\scriptstyle\frac12
\right]_{(\mathds{1}_1)},
\left[
\textstyle-s-\scriptstyle\frac32
\right]_{(\mathds{1}_1)},
\left[
\textstyle-s-\scriptstyle\frac52
\right]_{(\mathds{1}_1)}
\right),
$$ which is the unique irreducible quotient of $
\pi'
=
\mathcal{Q}\left(
\left[
\textstyle-s-\scriptstyle\frac12,~
\textstyle-s+\scriptstyle\frac12
\right]_{(\mathds{1}_1)}
\right)
\times
\nu_{\mathds{1}_1}^{-s-\scriptstyle\frac32}
\times
\nu_{\mathds{1}_1}^{-s-\scriptstyle\frac52}.
$ Since $\pi'$ has $\nu^{-2s}\St_2\times\nu^{-2(s+2)}$ as a quotient, $\theta_7$ is also the unique irreducible quotient of $\nu^{-2s}\St_2\times\nu^{-2(s+2)}$. The latter representation does not satisfies the conditions in \eqref{4.1} and hence does not admit a generalized linear period  with respect to $(\Ha_{1,3},\chi_s)$. Hence, neither does $\theta_7$. Finally, $\theta_8$ does not admit a generalized linear period with respect to $(\Ha_{1,3},\chi_s)$ by Theorem~\ref{anand}.
\item[\upshape(ii)]
Suppose that $s=-1$. Then
$
\pi
=
\nu_{\mathds{1}_1}^{-s+\frac12}
\times\nu_{\mathds{1}_1}^{-s-\frac32}
\times\nu_{\mathds{1}_1}^{-s-\frac12}
\times\nu_{\mathds{1}_1}^{-s-\frac12}
$
has the following Jordan--Hölder constituents:
\[
\begin{aligned}
\theta_1
&=
\mathcal{Q}\left(
\left[
\textstyle-s-\scriptstyle\frac32,~
\textstyle-s+\scriptstyle\frac12
\right]_{(\mathds{1}_1)},
\left[
\textstyle-s-\scriptstyle\frac12
\right]_{(\mathds{1}_1)}
\right),\\
\theta_2
&=
\mathcal{Q}\left(
\left[
\textstyle-s-\scriptstyle\frac12,~
\textstyle-s+\scriptstyle\frac12
\right]_{(\mathds{1}_1)},
\left[
\textstyle-s-\scriptstyle\frac32,~
\textstyle-s-\scriptstyle\frac12
\right]_{(\mathds{1}_1)}
\right),\\
\theta_3
&=
\mathcal{Q}\left(
\left[
\textstyle-s+\scriptstyle\frac12
\right]_{(\mathds{1}_1)},
\left[
\textstyle-s-\scriptstyle\frac32,~
\textstyle-s-\scriptstyle\frac12
\right]_{(\mathds{1}_1)},
\left[
\textstyle-s-\scriptstyle\frac12
\right]_{(\mathds{1}_1)}
\right),\\
\theta_4
&=
\mathcal{Q}\left(
\left[
\textstyle-s+\scriptstyle\frac12
\right]_{(\mathds{1}_1)},
\left[
\textstyle-s-\scriptstyle\frac32
\right]_{(\mathds{1}_1)},
\left[
\textstyle-s-\scriptstyle\frac12
\right]_{(\mathds{1}_1)},
\left[
\textstyle-s-\scriptstyle\frac12
\right]_{(\mathds{1}_1)}
\right).
\end{aligned}
\]
The representation $\theta_1$ is the unique irreducible quotient of
$
\nu_{\mathds{1}_1}^{-s-\scriptstyle\frac12}
\times
\mathcal{Q}\left(
\left[
\textstyle-s-\scriptstyle\frac32,~
\textstyle-s+\scriptstyle\frac12
\right]_{(\mathds{1}_1)}
\right),
$
which satisfies none of the conditions in Lemma~\ref{lem3} with $n=4$ and $k=1$, hence it does not admit a generalized linear period with respect to $(\Ha_{1,3},\chi_s)$. Hence, neither does $\theta_1$. Similarly, $\theta_2$ is the unique
irreducible quotient of
$$
\mathcal{Q}\left(
\left[
\textstyle-s-\scriptstyle\frac12,~
\textstyle-s+\scriptstyle\frac12
\right]_{(\mathds{1}_1)}
\right)
\times
\mathcal{Q}\left(
\left[
\textstyle-s-\scriptstyle\frac32,~
\textstyle-s-\scriptstyle\frac12
\right]_{(\mathds{1}_1)}
\right),
$$
which does not satisfies the conditions in \eqref{4.1}, and hence does not admit a generalized
linear period  with respect to $(\Ha_{1,3},\chi_s)$. Furthermore, $\theta_3$ is isomorphic to
$\nu^{-2s}\times\nu^{-2s-2}\St_2$. The question of whether this
representation admits a generalized linear period with respect to
$(\Ha_{1,3},\chi_s)$ was addressed in Case~\textup{(6)}.
 Finally, by M\oe glin--Waldspurger algorithm~\cite[\S 1]{BR07}, we have
$$
\theta_4\simeq
\mathcal{S}\left(
\left[
\textstyle-s-\scriptstyle\frac32,~
\textstyle-s+\scriptstyle\frac12
\right]_{(\mathds{1}_1)}
\right)
\times
\nu_{\mathds{1}_1}^{-s-\scriptstyle\frac12}.
$$
 Lemma~\ref{lem3} with $n=4$ and $k=3$ shows that $\theta_4$ admits a generalized
linear period  with respect to $(\Ha_{1,3},\chi_s)$ arising from the closed orbit.

\item[\upshape(iii)]
Suppose that $s=0$. Then
$
\pi
=
\nu_{\mathds{1}_1}^{-s+\frac12}
\times\nu_{\mathds{1}_1}^{-s-\frac32}
\times\nu_{\mathds{1}_1}^{-s+\frac32}
\times\nu_{\mathds{1}_1}^{-s-\frac12}
$
has the following Jordan--Hölder constituents:
\[
\begin{aligned}
\theta_1
&=
\mathcal{Q}\left(
\left[
\textstyle-s-\scriptstyle\frac32,~
\textstyle-s+\scriptstyle\frac32
\right]_{(\mathds{1}_1)}
\right),\\
\theta_2
&=
\mathcal{Q}\left(
\left[
\textstyle-s-\scriptstyle\frac32,~
\textstyle-s+\scriptstyle\frac12
\right]_{(\mathds{1}_1)},
\left[
\textstyle-s+\scriptstyle\frac32
\right]_{(\mathds{1}_1)}
\right),\\
\theta_3
&=
\mathcal{Q}\left(
\left[
\textstyle-s-\scriptstyle\frac32
\right]_{(\mathds{1}_1)},
\left[
\textstyle-s-\scriptstyle\frac12,~
\textstyle-s+\scriptstyle\frac32
\right]_{(\mathds{1}_1)}
\right),\\
\theta_4
&=
\mathcal{Q}\left(
\left[
\textstyle-s+\scriptstyle\frac12,~
\textstyle-s+\scriptstyle\frac32
\right]_{(\mathds{1}_1)},
\left[
\textstyle-s-\scriptstyle\frac32,~
\textstyle-s-\scriptstyle\frac12
\right]_{(\mathds{1}_1)}
\right),\\
\theta_5
&=
\mathcal{Q}\left(
\left[
\textstyle-s+\scriptstyle\frac12
\right]_{(\mathds{1}_1)},
\left[
\textstyle-s-\scriptstyle\frac32,~
\textstyle-s-\scriptstyle\frac12
\right]_{(\mathds{1}_1)},
\left[
\textstyle-s+\scriptstyle\frac32
\right]_{(\mathds{1}_1)}
\right),\\
\theta_6
&=
\mathcal{Q}\left(
\left[
\textstyle-s-\scriptstyle\frac32
\right]_{(\mathds{1}_1)},
\left[
\textstyle-s-\scriptstyle\frac12,~
\textstyle-s+\scriptstyle\frac12
\right]_{(\mathds{1}_1)},
\left[
\textstyle-s+\scriptstyle\frac32
\right]_{(\mathds{1}_1)}
\right),\\
\theta_7
&=
\mathcal{Q}\left(
\left[
\textstyle-s-\scriptstyle\frac32
\right]_{(\mathds{1}_1)},
\left[
\textstyle-s+\scriptstyle\frac12,~
\textstyle-s+\scriptstyle\frac32
\right]_{(\mathds{1}_1)},
\left[
\textstyle-s-\scriptstyle\frac12
\right]_{(\mathds{1}_1)}
\right),\\
\theta_8
&=
\mathcal{Q}\left(
\left[
\textstyle-s+\scriptstyle\frac12
\right]_{(\mathds{1}_1)},
\left[
\textstyle-s-\scriptstyle\frac32
\right]_{(\mathds{1}_1)},
\left[
\textstyle-s+\scriptstyle\frac32
\right]_{(\mathds{1}_1)},
\left[
\textstyle-s-\scriptstyle\frac12
\right]_{(\mathds{1}_1)}
\right).
\end{aligned}
\]
The representation $\theta_1$ does not admit a generalized linear period with respect to $(\Ha_{1,3},\chi_s)$ by Theorem~\ref{anand}. By Lemma~\ref{l1}, $\theta_2$ is the unique irreducible quotient of $$
\pi'
=
\nu_{\mathds{1}_1}^{-s+\scriptstyle\frac32}
\times
\mathcal{Q}\left(
\left[
\textstyle-s-\scriptstyle\frac32,~
\textstyle-s+\scriptstyle\frac12
\right]_{(\mathds{1}_1)}
\right),
$$ which does not satisfy the conditions of \eqref{4.1} with $n=4$ and $k=1$. Consequently, $\pi'$ does not admit a generalized linear period with respect to $(\Ha_{1,3},\chi_s)$, and hence neither does $\theta_2$. Taking the contragredient, we conclude that $\theta_3$ does not admit a generalized linear period with respect to $(\Ha_{1,3},\chi_s)$ either. The remaining Jordan--Hölder constituents $\theta_4,\ldots,\theta_8$ were treated in the preceding cases and yield no additional irreducible representations admitting a generalized linear period with respect to $(\Ha_{1,3},\chi_s)$ beyond those listed in Theorem~\ref{G4}. We therefore omit the repetitive individual analysis of these Jordan--Hölder constituents.
\end{enumerate}

\smallskip
\noindent\textbf{(8c)}
Suppose that $\sigma=\nu_{\mathds{1}_1}^{-s+\frac12}$ with
$s\neq\pm\frac12$. If $s\notin\{0,1\}$, then $\pi$ has the following Jordan--Hölder constituents:
\[
\begin{aligned}
\theta_1
&=
\nu_{\mathds{1}_1}^{s-\scriptstyle\frac12}
\times
\nu_{\mathds{1}_1}^{-s+\scriptstyle\frac12}
\times
\mathcal{Q}\left(
\left[
\textstyle-s-\scriptstyle\frac12,~
\textstyle-s+\scriptstyle\frac12
\right]_{(\mathds{1}_1)}
\right),\\
\theta_2
&=
\nu_{\mathds{1}_1}^{s-\scriptstyle\frac12}
\times
\nu_{\mathds{1}_1}^{-s+\scriptstyle\frac12}
\times
\mathcal{Q}\left(
\left[
\textstyle-s+\scriptstyle\frac12
\right]_{(\mathds{1}_1)},
\left[
\textstyle-s-\scriptstyle\frac12
\right]_{(\mathds{1}_1)}
\right).
\end{aligned}
\]
By Case~\textup{(5)}, $\theta_1$ does not admit a generalized
linear period with respect to $(\Ha_{1,3},\chi_s)$, whereas
Case~\textup{(7)} shows that $\theta_2$ does admit such a period.

In the case where $s=0$, $\pi$ has the following Jordan--Hölder constituents:
\[
\begin{aligned}
\theta_1
&=
\mathcal{Q}\left(
\left[
\textstyle-s-\scriptstyle\frac12,~
\textstyle-s+\scriptstyle\frac12
\right]_{(\mathds{1}_1)}
\right)
\times
\mathcal{Q}\left(
\left[
\textstyle-s-\scriptstyle\frac12,~
\textstyle-s+\scriptstyle\frac12
\right]_{(\mathds{1}_1)}
\right),\\
\theta_2
&=
\mathcal{Q}\left(
\left[
\textstyle-s+\scriptstyle\frac12
\right]_{(\mathds{1}_1)},
\left[
\textstyle-s-\scriptstyle\frac12
\right]_{(\mathds{1}_1)}
\right)
\times
\mathcal{Q}\left(
\left[
\textstyle-s-\scriptstyle\frac12,~
\textstyle-s+\scriptstyle\frac12
\right]_{(\mathds{1}_1)}
\right),\\
\theta_3
&=
\mathcal{Q}\left(
\left[
\textstyle-s+\scriptstyle\frac12
\right]_{(\mathds{1}_1)},
\left[
\textstyle-s-\scriptstyle\frac12
\right]_{(\mathds{1}_1)}
\right)
\times
\mathcal{Q}\left(
\left[
\textstyle-s+\scriptstyle\frac12
\right]_{(\mathds{1}_1)},
\left[
\textstyle-s-\scriptstyle\frac12
\right]_{(\mathds{1}_1)}
\right).
\end{aligned}
\]
The question of whether these representations admit a generalized linear period with respect to $(\Ha_{1,3},\chi_s)$ has already been addressed in the preceding cases.
If $s=1$, then $\pi$ has the following Jordan--Hölder constituents:
\[
\begin{aligned}
\theta_1
&=
\mathcal{Q}\left(
\left[
\textstyle-s-\scriptstyle\frac12,~
\textstyle-s+\scriptstyle\frac32
\right]_{(\mathds{1}_1)},
\left[
\textstyle-s+\scriptstyle\frac12
\right]_{(\mathds{1}_1)}
\right),\\
\theta_2
&=
\mathcal{Q}\left(
\left[
\textstyle-s+\scriptstyle\frac12,~
\textstyle-s+\scriptstyle\frac32
\right]_{(\mathds{1}_1)},
\left[
\textstyle-s-\scriptstyle\frac12,~
\textstyle-s+\scriptstyle\frac12
\right]_{(\mathds{1}_1)}
\right),\\
\theta_3
&=
\mathcal{Q}\left(
\left[
\textstyle-s+\scriptstyle\frac12
\right]_{(\mathds{1}_1)},
\left[
\textstyle-s-\scriptstyle\frac12,~
\textstyle-s+\scriptstyle\frac12
\right]_{(\mathds{1}_1)},
\left[
\textstyle-s+\scriptstyle\frac32
\right]_{(\mathds{1}_1)}
\right),\\
\theta_4
&=
\mathcal{Q}\left(
\left[
\textstyle-s+\scriptstyle\frac12
\right]_{(\mathds{1}_1)},
\left[
\textstyle-s+\scriptstyle\frac12
\right]_{(\mathds{1}_1)},
\left[
\textstyle-s+\scriptstyle\frac32
\right]_{(\mathds{1}_1)},
\left[
\textstyle-s-\scriptstyle\frac12
\right]_{(\mathds{1}_1)}
\right).
\end{aligned}
\]

The question of whether these representations admit a generalized linear period with respect to $(\Ha_{1,3},\chi_s)$  has already been discussed.

The remaining possibilities
$\sigma\in
\left\{
\nu_{\mathds{1}_1}^{\pm\frac12},
\nu_{\mathds{1}_1}^{s-\frac32},
\nu_{\mathds{1}_1}^{s+\frac12}
\right\}$
are treated in the same manner and yield no additional Jordan--Hölder constituents admitting a generalized linear period  with respect to $(\Ha_{1,3},\chi_s)$, up to taking the contragredient and replacing $s$ by $-s$.

Combining Cases~\textup{(1)}--\textup{(8)}, we obtain exactly the three families
listed in the statement. Conversely, each representation in those families
has been shown above to admit a generalized linear period with respect to
$(\Ha_{1,3},\chi_s)$. This proves the theorem and establishes
Conjecture~\ref{conj} for $n=4$.
\end{proof}

\section{Proof of the main theorems} \label{JLTLP}

      In this section, we prove Theorem~\ref{LP} and Theorem \ref{Lapid-Prasad}. We begin by recalling the Jacquet--Langlands transfer and the relevant facts about Langlands parameters that will be used in the proof.

\subsection{Jacquet--Langlands transfer }\label{JLT}

      In this subsection, we briefly recall the Jacquet--Langlands transfer and refer to \cite[\S 2]{AIB08} and \cite[\S 2.3]{Suz20} for details. The Jacquet--Langlands transfer (denoted by $\JL$) gives a bijection between the irreducible essentially square-integrable representations of $\G_n$ and those of $\GL_{2n}(\F)$. If $\rho\in\mathcal{C}(\G_k)$, then there exist a positive integer $l_\rho$ dividing $2k$ and an irreducible supercuspidal representation $\rho'$ of $\GL_{2k/l_\rho}(\F)$ such that  \[\JL(\rho)=\mathcal{Q}\left(\left[-\scriptstyle\frac{l_\rho-1}{2},\scriptstyle\frac{l_\rho-1}{2}\right]_{(\rho')}\right).\]

      For a segment $\Delta=[a,b]_{(\rho)}$ of $\G_n$, define the segment $\JL(\Delta)$ of $\GL_{2n}(\F)$ by $\JL(\mathcal{Q}(\Delta))=\mathcal{Q}(\JL(\Delta))$. Explicitly, $$\JL(\Delta)=\left[\scriptstyle\frac{(2a-1)l_\rho+1}{2},\scriptstyle\frac{(2b+1)l_\rho-1}{2}\right]_{(\rho')}.$$
      Moreover, for any two segments $\Delta$ and $\Delta'$, $\Delta$ precedes $\Delta'$ if and only if $\JL(\Delta)$ precedes $\JL(\Delta')$.
      If $\mathfrak m=\{\Delta_1,\ldots,\Delta_t\}$ is a multisegment and $\pi=\mathcal{Q}(\mathfrak m)$, define $\JL(\mathfrak m)=\{\JL(\Delta_1),\ldots,\JL(\Delta_t)\}$ and
      \begin{equation}\label{JLeqn}\JL(\pi)=\mathcal{Q}\bigl(\JL(\mathfrak m)\bigr).\end{equation} 
      
      \begin{definition}
An irreducible smooth representation $\pi$ of $\G_n$ is called
generic if its Jacquet--Langlands transfer
$\JL(\pi)$ is a generic representation of $\GL_{2n}(\F)$.
\end{definition}

Note that, by the classification of generic representations, $\JL(\pi)$ is generic if and only if the multisegment $\mathfrak m$ is pairwise unlinked.

\subsection{Langlands parameters}\label{LPr}

\noindent
This section summarizes the relevant results (\hspace{-0.01cm}\cite{kudla, DPR}) on the
Langlands parameters of certain irreducible admissible representations
of $\GL_n(\F)$. Let $W_{\F}$ denote the Weil group of $\F$, and let
$W_{\F}'=W_{\F}\times\SL(2,\mathbb{C})$ be the Weil--Deligne group of
$\F$. The local Langlands correspondence gives a canonical bijection
between $\Irr(\GL_n(\F))$ and the $n$-dimensional representations of
$W_{\F}'$. We denote this correspondence by $\mathfrak{L}$ and for
$\pi\in\Irr(\GL_n(\F))$, call $\mathfrak{L}(\pi)$ the
\emph{Langlands parameter} of $\pi$. Thus,
$$\mathfrak{L}(\pi):W_{\F}'\to\GL_n(\mathbb{C})$$ is a continuous
homomorphism such that $\mathfrak{L}(\pi)(W_{\F})$ consists of
semisimple elements and
$\mathfrak{L}(\pi)|_{\SL(2,\mathbb{C})}$ is a morphism of complex
algebraic groups. We say that $\mathfrak{L}(\pi)$ is
\emph{symplectic}, and correspondingly that $\pi$ is of
\emph{symplectic type}, if $\mathfrak{L}(\pi)$ preserves a
nondegenerate symplectic form. Equivalently, the image of
$\mathfrak{L}(\pi)$ is contained in the symplectic group
$\Sp_n(\mathbb{C})\subset\GL_n(\mathbb{C})$. Every $n$-dimensional
representation $\mathfrak{L}$ of $W_{\F}'$ decomposes as
$$
\mathfrak{L}=\bigoplus_{i=1}^{r}\mathfrak{L}_i\otimes\Sp_{r_i},
$$
where each $\mathfrak{L}_i$ is an irreducible representation of $W_{\F}$ of
dimension $n_i$ and $\Sp_{r_i}$ is the unique irreducible
representation of $\SL(2,\mathbb{C})$ of dimension $r_i$ with
$\sum_{i=1}^r n_i r_i=n$. Under the local Langlands correspondence,
$\Sp_{r_i}$ corresponds to the Steinberg representation
$\St_{r_i}$ of $\GL_{r_i}(\F)$. Furthermore,
under the local Langlands correspondence, every irreducible
supercuspidal representation of $\GL_n(\F)$ corresponds to an
irreducible $n$-dimensional representation of $W_{\F}$.

Assuming that the Langlands parameters of supercuspidal
representations are known, we now recall how to determine the
Langlands parameter of an irreducible admissible representation of
$\GL_n(\F)$. This follows immediately from the Langlands
classification. According to the Langlands classification, the
representation $\pi=\mathcal{Q}(\Delta_1,\ldots,\Delta_k)$ is the
unique irreducible quotient of the induced representation
$\mathcal{Q}(\Delta_1)\times\cdots\times\mathcal{Q}(\Delta_k)$,
provided that the segments $\Delta_i$ are ordered so that $\Delta_i$
does not precede $\Delta_j$ for all $i<j$. Suppose that each
$\Delta_i$ is the segment $[0,{l_i-1}]_{(\sigma_i)}$ of length $l_i$, where
$\sigma_i$ is an irreducible supercuspidal representation of
$\GL_{m_i}(\F)$. Then the Langlands parameter of $\pi$ is given by
$$
\mathfrak{L}(\pi)
=
\bigoplus_{i=1}^{k}
\left(\mathfrak{L}(\sigma_i)\otimes\nu^{\frac{l_i-1}{2}}\Sp_{l_i}\right).
$$

We now recall the following lemma concerning the Langlands parameter
of an irreducible parabolically induced representation, which follows
from \cite[Corollary~A.6(ii)]{Tad86}.

\begin{lemma}\label{LPLemma}
Let $\pi_1\in\Irr(\GL_{n_1}(\F))$ and
$\pi_2\in\Irr(\GL_{n_2}(\F))$. If $\pi_1\times\pi_2$ is irreducible, then
\[
\mathfrak{L}(\pi_1\times\pi_2)
=
\mathfrak{L}(\pi_1)\oplus\mathfrak{L}(\pi_2).
\]
\end{lemma}

We next recall the Langlands parameter associated with an irreducible representation of $\G_n$ via the Jacquet--Langlands correspondence, as defined in \cite[Definition 2.1]{anandavardhanan2024sign}.

\begin{definition}\label{def:LP}
 The Langlands parameter of an irreducible smooth representation $\pi$ of $\G_n$ is defined by 
\[
\mathfrak{L}(\pi):=\mathfrak{L}(\JL(\pi)),
\]
where $\mathfrak{L}(\JL(\pi))$ denotes the Langlands parameter of $\JL(\pi)$ under the local Langlands correspondence for $\GL_{2n}(\F)$.
\end{definition}

\subsection{Proof of Theorem~\ref{LP}}




\begin{proof}
Throughout the proof, the group on which a power of $\nu$ acts is
determined by the context. We use $\mathds{1}_1$ for the trivial
representation of $\G_1$ and $\mathds{1}$ for the trivial character of
$\GL_1(\F)$.

We first record the relation between multisegments, the
Jacquet--Langlands transfer, and Langlands parameters that will be
used repeatedly.

\medskip

\noindent\textbf{The multisegment-to-parameter calculation.}
For $t\in\mathbb R$, define the multisegment
$$
\mathfrak c_m(t)
=
\left\{
\left[
\frac t2+\frac{m-1}{2}-j
\right]_{(\mathds{1}_1)}
:
0\leq j\leq m-1
\right\}.
$$

The character $\nu^t$ of $\G_m$ is given in the Langlands
classification by
$
\nu^t
=
\mathcal Q\bigl(\mathfrak c_m(t)\bigr).
$
Indeed, since $\nu_{\mathds{1}_1}=\nu^2$, twisting the multisegment of
$\mathds{1}_m$ by $\nu^t$ shifts each exponent by $\frac t2$.
For the cuspidal representation $\mathds{1}_1$ of $\G_1$, we have
$l_{\mathds{1}_1}=2$ and
$$
\JL(\mathds{1}_1)
=
\mathcal Q
\left(
\left[-\frac12,\frac12\right]_{(\mathds{1})}
\right).
$$
Consequently, for a segment $[a,b]_{(\mathds{1}_1)}$, the
Jacquet--Langlands transfer is
$$
\JL\left([a,b]_{(\mathds{1}_1)}\right)
=
\left[
2a-\frac12,\,
2b+\frac12
\right]_{(\mathds{1})}.
$$
The segment on the right has length
$
2(b-a+1)
$
and center $a+b$. Therefore, under the local Langlands
correspondence for the split general linear group,
\begin{equation}\label{5.1}
\mathfrak L
\left(
\mathcal Q\left([a,b]_{(\mathds{1}_1)}\right)
\right)
\simeq
\nu^{a+b}\Sp_{2(b-a+1)}.
\tag{5.1}  
\end{equation}
In particular, for a singleton segment $[a]_{(\mathds{1}_1)}$,
\begin{equation}\label{5.2}
\mathfrak L
\left(
\mathcal Q\left([a]_{(\mathds{1}_1)}\right)
\right)
\simeq
\nu^{2a}\Sp_2.
\tag{5.2}
\end{equation}

More generally, if
$
\mathfrak m
=
\left\{
[a_1,b_1]_{(\mathds{1}_1)},\ldots,
[a_r,b_r]_{(\mathds{1}_1)}
\right\},
$
then the Jacquet--Langlands transfer is obtained by applying $\JL$
to each segment:
$$
\JL\bigl(\mathcal Q(\mathfrak m)\bigr)
=
\mathcal Q
\left(
\left\{
\left[
2a_i-\frac12,\,
2b_i+\frac12
\right]_{(\mathds{1})}
\right\}_{i=1}^{r}
\right).
$$
It follows from the Langlands classification for $\GL_{2n}(\F)$ that
\begin{equation}\label{5.3}
\mathfrak L\bigl(\mathcal Q(\mathfrak m)\bigr)
\simeq
\bigoplus_{i=1}^{r}
\nu^{a_i+b_i}\Sp_{2(b_i-a_i+1)}.
\tag{5.3}
\end{equation}
Applying this formula to $\mathfrak c_m(t)$ gives
$
\mathfrak L(\nu^t)
\simeq
\bigoplus_{j=0}^{m-1}
\nu^{t+m-1-2j}\Sp_2.
$
In particular, the character $\nu^{-2s}$ of $\G_{n-2}$ corresponds
to the multisegment
$$
\mathfrak a_s
=
\mathfrak c_{n-2}(-2s)
=
\left\{
\left[
-s+\frac{n-3}{2}-j
\right]_{(\mathds{1}_1)}
:
0\leq j\leq n-3
\right\},
$$
and its Langlands parameter is
\begin{equation}\label{5.4}
\mathfrak L(\nu^{-2s})
\simeq
\bigoplus_{j=0}^{n-3}
\nu^{-2s+n-3-2j}\Sp_2.
\tag{5.4}
\end{equation}
Since every summand in \eqref{5.4} has dimension $2$, this parameter has
dimension $2n-4$.

We now prove the two implications.

\medskip

\noindent\textbf{Forward implication.}
Assume that $\pi$ admits a generalized linear period with respect to
$(\Ha_{1,n-1},\chi_s)$. By Conjecture~\ref{conj}, the
representation $\pi$ belongs to one of the families appearing in
parts~\textup{(1)}--\textup{(3)}. We consider all possibilities.

\medskip

\noindent\emph{Case 1: $\pi=\mathds{1}_n$ and $s=0$.}
At $s=0$, the unique irreducible subrepresentations in
Conjecture~\ref{conj}\textup{(2)} and
Conjecture~\ref{conj}\textup{(3)} are both isomorphic to
$\mathds{1}_n$. Its Langlands multisegment is
$$
\mathfrak c_n(0)
=
\left\{
\left[
\frac{n-1}{2}-j
\right]_{(\mathds{1}_1)}
:
0\leq j\leq n-1
\right\},
$$
so that
$
\mathds{1}_n
=
\mathcal Q\bigl(\mathfrak c_n(0)\bigr).
$
Applying the Jacquet--Langlands transfer gives
$$
\JL(\mathds{1}_n)
=
\mathcal Q
\left(
\left\{
\left[
n-\frac32-2j,\,
n-\frac12-2j
\right]_{(\mathds{1})}
:
0\leq j\leq n-1
\right\}
\right).
$$
Hence, by \eqref{5.3},
$
\mathfrak L(\mathds{1}_n)
\simeq
\bigoplus_{j=0}^{n-1}
\nu^{n-1-2j}\Sp_2.
$
Equivalently,
$$
\mathfrak L(\mathds{1}_n)
\simeq
\nu^{n-1}\Sp_2
\oplus
\nu^{n-3}\Sp_2
\oplus\cdots\oplus
\nu^{-n+3}\Sp_2
\oplus
\nu^{-n+1}\Sp_2.
$$
The middle $n-2$ summands form
$
\mathfrak L(\mathds{1}_{n-2})
=
\mathfrak L(\nu^{-2s}),
$
because $s=0$. Therefore,
\begin{equation}\label{5.5}
\mathfrak L(\pi)/\mathfrak L(\nu^{-2s})
\simeq
\nu^{n-1}\Sp_2
\oplus
\nu^{-n+1}\Sp_2.
\tag{5.5}
\end{equation}

Now set
$
\tau_0
=
\nu^{n-1}
\times
\nu^{-n+1}.
$
The multisegment of $\tau_0$ is
$
\left\{
\left[\frac{n-1}{2}\right]_{(\mathds{1}_1)},
\left[-\frac{n-1}{2}\right]_{(\mathds{1}_1)}
\right\}.
$
Thus, \eqref{5.2} gives
$$
\mathfrak L(\tau_0)
\simeq
\nu^{n-1}\Sp_2
\oplus
\nu^{-n+1}\Sp_2.
$$
Since $n>2$, the two singleton segments are not linked, and
$\tau_0$ is irreducible and infinite-dimensional. It is
$\Ha_{1,1}$-distinguished because it is of the form
$\sigma\times\widetilde{\sigma}$. It follows from \eqref{5.5} that
$$
\mathfrak L(\pi)/\mathfrak L(\nu^{-2s})
\simeq
\mathfrak L(\tau_0).
$$
Thus, at $s=0$, the four-dimensional quotient belongs to the second
alternative of the theorem.

\medskip

\noindent\emph{Case 2:
$\pi=\nu^{-2s}\times\tau$, where $\tau\in\Irr(\G_2)$ be infinite-dimensional and
$\Ha_{1,1}$-distinguished.} Write
$
\tau=\mathcal Q(\mathfrak b)
$
for its Langlands multisegment $\mathfrak b$. Since $\pi$ irreducible,
$
\pi
=
\mathcal Q(\mathfrak a_s\cup\mathfrak b).
$
The Jacquet--Langlands transfer is
$
\JL(\pi)
=
\mathcal Q
\left(
\JL(\mathfrak a_s)\cup\JL(\mathfrak b)
\right).
$
Applying the split local Langlands correspondence and using \eqref{5.3}
gives
$$
\mathfrak L(\pi)
\simeq
\mathfrak L(\nu^{-2s})
\oplus
\mathfrak L(\tau).
$$
Therefore,
$$
\mathfrak L(\pi)/\mathfrak L(\nu^{-2s})
\simeq
\mathfrak L(\tau),
$$
which is the second alternative of the theorem.

\medskip

\noindent\emph{Case 3: 
$\pi
=
\nu_{\mathds{1}_1}^{s+\frac{n-1}{2}}
\times
\nu^{-2s-1}~\text{ with}~
s\notin\left\{-\frac n2,0\right\}$.}
The first factor has multisegment
$
\left\{
\left[
s+\frac{n-1}{2}
\right]_{(\mathds{1}_1)}
\right\}.
$
The second factor is a character of $\G_{n-1}$ and has
multisegment
$$
\mathfrak c_{n-1}(-2s-1)
=
\left\{
\left[
-s+\frac{n-3}{2}-j
\right]_{(\mathds{1}_1)}
:
0\leq j\leq n-2
\right\}.
$$
Consequently, the multisegment of $\pi$ is
\begin{equation}\label{5.6}
\mathfrak m_s^{+}
=
\left\{
\left[
s+\frac{n-1}{2}
\right]_{(\mathds{1}_1)}
\right\}
\cup
\left\{
\left[
-s+\frac{n-3}{2}-j
\right]_{(\mathds{1}_1)}
:
0\leq j\leq n-2
\right\}.
\tag{5.6}
\end{equation}
Applying $\JL$ to each singleton segment gives
$$
\begin{aligned}
\JL(\pi)
=
\mathcal Q\Bigg(
&
\left[
2s+n-\frac32,\,
2s+n-\frac12
\right]_{(\mathds{1})},\\
&
\left\{
\left[
-2s+n-\frac72-2j,\,
-2s+n-\frac52-2j
\right]_{(\mathds{1})}
:
0\leq j\leq n-2
\right\}
\Bigg).
\end{aligned}
$$
Therefore,
$$
\begin{aligned}
\mathfrak L(\pi)
\simeq{}
\nu^{2s+n-1}\Sp_2
\oplus
\nu^{-2s+n-3}\Sp_2
\oplus\cdots\oplus
\nu^{-2s-n+3}\Sp_2
\oplus
\nu^{-2s-n+1}\Sp_2.
\end{aligned}
$$
The summands
$
\nu^{-2s+n-3}\Sp_2
\oplus\cdots\oplus
\nu^{-2s-n+3}\Sp_2
$
form $\mathfrak L(\nu^{-2s})$. Hence,
\begin{equation}\label{5.7}
\mathfrak L(\pi)/\mathfrak L(\nu^{-2s})
\simeq
\nu^{2s+n-1}\Sp_2
\oplus
\nu^{-2s-n+1}\Sp_2.
\tag{5.7}
\end{equation}

Set
$$
\tau_s^{+}
=
\nu_{\mathds{1}_1}^{s+\frac{n-1}{2}}
\times
\nu_{\mathds{1}_1}^{-\left(s+\frac{n-1}{2}\right)}.
$$
Its multisegment consists of the two singleton segments
$
\left[
s+\frac{n-1}{2}
\right]_{(\mathds{1}_1)}
~\text{and}~
\left[
-s-\frac{n-1}{2}
\right]_{(\mathds{1}_1)}.
$
Thus,
$$
\mathfrak L(\tau_s^{+})
\simeq
\nu^{2s+n-1}\Sp_2
\oplus
\nu^{-2s-n+1}\Sp_2.
$$

If $s=-\frac{n-2}{2}$, then
$$
\mathfrak L(\pi)/\mathfrak L(\nu^{-2s})
\simeq
\nu\Sp_2\oplus\nu^{-1}\Sp_2.
$$
The trivial representation $\mathds{1}_2$ has Langlands
multisegment
$
\left\{
\left[\frac12\right]_{(\mathds{1}_1)},
\left[-\frac12\right]_{(\mathds{1}_1)}
\right\},
$
and therefore
$$
\mathfrak L(\mathds{1}_2)
\simeq
\nu\Sp_2\oplus\nu^{-1}\Sp_2.
$$
Thus, at $s=-\frac{n-2}{2}$, the quotient in \eqref{5.7} is
$\mathfrak L(\mathds{1}_2)$.

For all other values occurring in this case, the two segments
defining $\tau_s^{+}$ are not linked. Hence, $\tau_s^{+}$ is an
irreducible infinite-dimensional representation. Since it is of the
form $\sigma\times\widetilde{\sigma}$, it is
$\Ha_{1,1}$-distinguished. Therefore, the quotient in \eqref{5.7} is
$\mathfrak L(\tau_s^{+})$.

\medskip

\noindent\emph{Case 4:
$
\pi
=
\nu_{\mathds{1}_1}^{s-\frac{n-1}{2}}
\times
\nu^{-2s+1}~\text{with}~
s\notin\left\{0,\frac n2\right\}.
$}
Its multisegment is
\begin{equation}\label{5.8}
\mathfrak m_s^{-}
=
\left\{
\left[
s-\frac{n-1}{2}
\right]_{(\mathds{1}_1)}
\right\}
\cup
\left\{
\left[
-s+\frac{n-1}{2}-j
\right]_{(\mathds{1}_1)}
:
0\leq j\leq n-2
\right\}.
\tag{5.8}
\end{equation}
Applying the Jacquet--Langlands transfer gives
$$
\begin{aligned}
\JL(\pi)
=
\mathcal Q\Bigg(
&
\left[
2s-n+\frac12,\,
2s-n+\frac32
\right]_{(\mathds{1})},\\
&
\left\{
\left[
-2s+n-\frac32-2j,\,
-2s+n-\frac12-2j
\right]_{(\mathds{1})}
:
0\leq j\leq n-2
\right\}
\Bigg).
\end{aligned}
$$
It follows that
$$
\begin{aligned}
\mathfrak L(\pi)
\simeq{}
\nu^{2s-n+1}\Sp_2
\oplus
\nu^{-2s+n-1}\Sp_2
\oplus
\nu^{-2s+n-3}\Sp_2
\oplus\cdots\oplus
\nu^{-2s-n+3}\Sp_2.
\end{aligned}
$$
Removing the summands constituting $\mathfrak L(\nu^{-2s})$, we
obtain
$$
\mathfrak L(\pi)/\mathfrak L(\nu^{-2s})
\simeq
\nu^{2s-n+1}\Sp_2
\oplus
\nu^{-2s+n-1}\Sp_2.
$$

Set
$
\tau_s^{-}
=
\nu_{\mathds{1}_1}^{s-\frac{n-1}{2}}
\times
\nu_{\mathds{1}_1}^{-\left(s-\frac{n-1}{2}\right)}.
$
By \eqref{5.2}, we have
$
\mathfrak L(\tau_s^{-})
\simeq
\nu^{2s-n+1}\Sp_2
\oplus
\nu^{-2s+n-1}\Sp_2.
$

If $s=\frac{n-2}{2}$, then

$$
\mathfrak L(\pi)/\mathfrak L(\nu^{-2s})
\simeq
\nu^{-1}\Sp_2\oplus\nu\Sp_2
\simeq
\mathfrak L(\mathds{1}_2).
$$
For every other value occurring in this case, $\tau_s^{-}$ is an
irreducible infinite-dimensional representation of the form
$\sigma\times\widetilde{\sigma}$ and is therefore
$\Ha_{1,1}$-distinguished.

\medskip

\noindent\emph{Case 5:
$\pi=\nu^{-2s}\mathcal Q_n$ and $s=\frac n2$.}
From Remark~\ref{LnZn}, the Langlands multisegment of $\mathcal Q_n$
is
$$
\left\{
\left[-\frac{n-3}{2}\right]_{(\mathds{1}_1)},
\dots,
\left[\frac{n-3}{2}\right]_{(\mathds{1}_1)},
\left[
\frac{n-1}{2},\frac{n+1}{2}
\right]_{(\mathds{1}_1)}
\right\}.
$$
Twisting by $\nu^{-2s}$ shifts every endpoint by $-s$. Thus,
\begin{equation}\label{5.9}
\begin{aligned}
\mathfrak q_s^{+}
=
\Bigg\{
\left[-s-\frac{n-3}{2}\right]_{(\mathds{1}_1)},
\dots,
\left[-s+\frac{n-3}{2}\right]_{(\mathds{1}_1)},
\left[
-s+\frac{n-1}{2},
-s+\frac{n+1}{2}
\right]_{(\mathds{1}_1)}
\Bigg\}
\end{aligned}
\tag{5.9}
\end{equation}
is the multisegment of $\pi$.

The singleton segments in \eqref{5.9} contribute
$$
\bigoplus_{j=0}^{n-3}
\nu^{-2s+n-3-2j}\Sp_2
=
\mathfrak L(\nu^{-2s}).
$$
For the final segment in \eqref{5.9}, formula \eqref{5.1} gives
$
\nu^{-2s+n}\Sp_4.
$
Since $s=\frac n2$, this is $\Sp_4$. Hence,
$$
\mathfrak L(\pi)
\simeq
\mathfrak L(\nu^{-2s})\oplus\Sp_4.
$$
Therefore,
$$
\mathfrak L(\pi)/\mathfrak L(\nu^{-2s})
\simeq
\Sp_4
\simeq
\mathfrak L(\St_2).
$$
The representation $\St_2$ is infinite-dimensional and
$\Ha_{1,1}$-distinguished.

\medskip

\noindent\emph{Case 6:
$\pi=\nu^{-2s}\widetilde{\mathcal Q}_n$ and
$s=-\frac n2$.}
By Remark~\ref{LnZn}, the Langlands multisegment of $\widetilde{\mathcal Q}_n$ is
$$
\left\{
\left[-\frac{n-3}{2}\right]_{(\mathds{1}_1)},
\dots,
\left[\frac{n-3}{2}\right]_{(\mathds{1}_1)},
\left[
-\frac{n+1}{2},-\frac{n-1}{2}
\right]_{(\mathds{1}_1)}
\right\}.
$$
After twisting by $\nu^{-2s}$, we obtain
$$
\begin{aligned}
\mathfrak q_s^{-}
=
\Bigg\{
\left[-s-\frac{n-3}{2}\right]_{(\mathds{1}_1)},
\dots,
\left[-s+\frac{n-3}{2}\right]_{(\mathds{1}_1)},
\left[
-s-\frac{n+1}{2},
-s-\frac{n-1}{2}
\right]_{(\mathds{1}_1)}
\Bigg\}.
\end{aligned}$$
The singleton segments again contribute
$\mathfrak L(\nu^{-2s})$. The last segment contributes
$\nu^{-2s-n}\Sp_4.$
Since $s=-\frac n2$, this equals $\Sp_4$. Therefore,
$
\mathfrak L(\pi)
\simeq
\mathfrak L(\nu^{-2s})\oplus\Sp_4,
$
and hence
$$
\mathfrak L(\pi)/\mathfrak L(\nu^{-2s})
\simeq
\Sp_4
\simeq
\mathfrak L(\St_2).
$$

This proves the forward implication.

\medskip

\noindent\textbf{Converse implication.}
Assume now that $\mathfrak L(\pi)$ contains a subrepresentation
isomorphic to $\mathfrak L(\nu^{-2s})$ and that the corresponding
four-dimensional quotient satisfies one of the two alternatives in
the statement.
The map
\[
\Irr(\G_n)
\longrightarrow
\left\{\text{$2n$-dimensional Weil--Deligne parameters}\right\},
\]
\[
\pi\longmapsto\mathfrak L(\JL(\pi))
\]
is injective. Indeed, $\JL$ is defined injectively on Langlands
multisegments, and the local Langlands correspondence for
$\GL_{2n}(\F)$ is a bijection. Consequently, once we identify an
irreducible representation having the prescribed parameter, it must
be isomorphic to $\pi$.
We treat the two possible four-dimensional quotients separately.

\medskip

\noindent\emph{Converse Case A: the quotient is
$\mathfrak L(\mathds{1}_2)$.}
We have
$
\mathfrak L(\mathds{1}_2)
\simeq
\nu\Sp_2\oplus\nu^{-1}\Sp_2.
$
By the hypothesis of the theorem, either
$s=-\frac{n-2}{2}$ or $s=\frac{n-2}{2}$.

Suppose first that
$
s=-\frac{n-2}{2}.
$
Consider the multisegment $\mathfrak m_s^{+}$ defined in \eqref{5.6}. The multisegment
$$
\left\{
\left[
-s+\frac{n-3}{2}-j
\right]_{(\mathds{1}_1)}
:
0\leq j\leq n-3
\right\}
$$
is precisely $\mathfrak a_s$. The two remaining singleton segments
are
$$
\left[
s+\frac{n-1}{2}
\right]_{(\mathds{1}_1)}
=
\left[\frac12\right]_{(\mathds{1}_1)}
~\text{
and}~
\left[
-s-\frac{n-1}{2}
\right]_{(\mathds{1}_1)}
=
\left[-\frac12\right]_{(\mathds{1}_1)}.
$$
By \eqref{5.2}, these contribute
$\nu\Sp_2\oplus\nu^{-1}\Sp_2$. Hence,
$
\mathfrak L\bigl(\mathcal Q(\mathfrak m_s^{+})\bigr)
\simeq
\mathfrak L(\nu^{-2s})
\oplus
\mathfrak L(\mathds{1}_2).
$
On the other hand,
$$
\mathcal Q(\mathfrak m_s^{+})
\simeq
\nu_{\mathds{1}_1}^{s+\frac{n-1}{2}}
\times
\nu^{-2s-1}.
$$
At $s=-\frac{n-2}{2}$, this induced representation is irreducible,
because this value is different from $-\frac n2$ and $0$.
Injectivity of the parameterization therefore gives
$$
\pi
\simeq
\nu_{\mathds{1}_1}^{s+\frac{n-1}{2}}
\times
\nu^{-2s-1}.
$$
This is the representation in Conjecture~\ref{conj}\textup{(2)}.
Hence, it admits a generalized linear period with respect to
$(\Ha_{1,n-1},\chi_s)$.

Now suppose that
$
s=\frac{n-2}{2}.
$
Consider the multisegment $\mathfrak m_s^{-}$ defined in \eqref{5.8}. It
contains $\mathfrak a_s$, and its two remaining singleton segments
are
$
\left[
s-\frac{n-1}{2}
\right]_{(\mathds{1}_1)}
=
\left[-\frac12\right]_{(\mathds{1}_1)}
$
and
$
\left[
-s+\frac{n-1}{2}
\right]_{(\mathds{1}_1)}
=
\left[\frac12\right]_{(\mathds{1}_1)}.
$
Therefore,
$
\mathfrak L\bigl(\mathcal Q(\mathfrak m_s^{-})\bigr)
\simeq
\mathfrak L(\nu^{-2s})
\oplus
\mathfrak L(\mathds{1}_2).
$
Moreover,
$$
\mathcal Q(\mathfrak m_s^{-})
\simeq
\nu_{\mathds{1}_1}^{s-\frac{n-1}{2}}
\times
\nu^{-2s+1}.
$$
This induced representation is irreducible because
$s\notin\{0,\frac n2\}$. Injectivity now implies that
$$
\pi
\simeq
\nu_{\mathds{1}_1}^{s-\frac{n-1}{2}}
\times
\nu^{-2s+1}.
$$
This representation occurs in Conjecture~\ref{conj}\textup{(3)} and
therefore admits the required generalized linear period.
This proves the converse when the quotient is
$\mathfrak L(\mathds{1}_2)$.

\medskip

\noindent\emph{Converse Case B: the quotient is
$\mathfrak L(\tau)$.}
Suppose that
\begin{equation}\label{5.10}
\mathfrak L(\pi)
\simeq
\mathfrak L(\nu^{-2s})
\oplus
\mathfrak L(\tau),
\tag{5.10}
\end{equation}
where $\tau\in\Irr(\G_2)$ is infinite-dimensional and
$\Ha_{1,1}$-distinguished. Write
$
\tau=\mathcal Q(\mathfrak b).
$
By \eqref{5.3},
$
\mathfrak L\bigl(\mathcal Q(\mathfrak a_s\cup\mathfrak b)\bigr)
\simeq
\mathfrak L(\nu^{-2s})
\oplus
\mathfrak L(\tau).
$
The injectivity of the Langlands parameterization therefore gives
\begin{equation}\label{5.11}
\pi
\simeq
\mathcal Q(\mathfrak a_s\cup\mathfrak b).
\tag{5.11}
\end{equation}

If $\nu^{-2s}\times\tau$ is irreducible, then
$
\mathcal Q(\mathfrak a_s\cup\mathfrak b)
\simeq
\nu^{-2s}\times\tau.
$
Thus,
$
\pi\simeq\nu^{-2s}\times\tau,
$
which is the representation in Conjecture~\ref{conj}\textup{(1)}.
It therefore admits the required generalized linear period with respect to $(\Ha_{1,n-1},\chi_s)$.

It remains to analyze the cases in which
$\nu^{-2s}\times\tau$ is reducible. Since the multisegment
$\mathfrak a_s$ is supported on the cuspidal line of
$\mathds{1}_1$, reducibility can occur only if the Langlands
multisegment of $\tau$ has a segment on the same cuspidal line that
interacts with an endpoint of $\mathfrak a_s$.
By the classification of the infinite-dimensional
$\Ha_{1,1}$-distinguished representations of $\G_2$, there are two
possibilities that require further consideration:

\begin{enumerate}
\item $\tau$ is an irreducible principal series representation
$\sigma\times\widetilde{\sigma}$ supported on the cuspidal line of
$\mathds{1}_1$;

\item $\tau=\St_2$.
\end{enumerate}

All the remaining distinguished representations are either supported
on a cuspidal line different from that of $\mathds{1}_1$, or their
segments are not linked with the endpoint segments of
$\mathfrak a_s$. Lemma~\ref{l1} then shows that
$\nu^{-2s}\times\tau$ is irreducible, so they have already been
covered by Conjecture~\ref{conj}\textup{(1)}.

\medskip

\noindent\emph{Converse Case B1: $\tau$ is principal series representation.}
Suppose that
$
\tau\simeq\sigma\times\widetilde{\sigma}
$
is an irreducible principal-series representation. If $\sigma$ is
not on the cuspidal line of $\mathds{1}_1$, then no segment in the
multisegment of $\tau$ is linked with a segment in $\mathfrak a_s$.
Thus, the induction is irreducible, and $\pi$ belongs to
Conjecture~\ref{conj}\textup{(1)}.
Suppose, therefore, that
$
\sigma=\nu_{\mathds{1}_1}^{x}
$
for some $x\in\mathbb R$. Then
$
\mathfrak b
=
\left\{
[x]_{(\mathds{1}_1)},
[-x]_{(\mathds{1}_1)}
\right\}.
$
The singleton exponents occurring in $\mathfrak a_s$ are
$$
-s+\frac{n-3}{2},\,
-s+\frac{n-5}{2},\,
\dots,\,
-s-\frac{n-3}{2}.
$$
The two exponents $x$ and $-x$ can extend this consecutive string only
at one of its two endpoints. Therefore, the linked-segment algorithm
gives
$$
x\in
\left\{
s+\frac{n-1}{2},
s-\frac{n-1}{2},
-s+\frac{n-1}{2},
-s-\frac{n-1}{2}
\right\}.
$$
Since replacing $x$ by $-x$ does not change the isomorphism class of
$\tau$, there are only two essentially different possibilities:
$$
\tau=\tau_s^{+}
=
\nu_{\mathds{1}_1}^{s+\frac{n-1}{2}}
\times
\nu_{\mathds{1}_1}^{-\left(s+\frac{n-1}{2}\right)}
~\text{and}~
\tau=\tau_s^{-}
=
\nu_{\mathds{1}_1}^{s-\frac{n-1}{2}}
\times
\nu_{\mathds{1}_1}^{-\left(s-\frac{n-1}{2}\right)}.
$$

Suppose first that $\tau=\tau_s^{+}$. Then
$
\mathfrak a_s\cup\mathfrak b
=
\mathfrak m_s^{+},
$
where $\mathfrak m_s^{+}$ is the multisegment in \eqref{5.6}. Hence, by
\eqref{5.11},
$
\pi\simeq\mathcal Q(\mathfrak m_s^{+}).
$
The representation $\tau_s^{+}$ is irreducible precisely when
$
s\notin
\left\{
-\frac n2,-\frac{n-2}{2}
\right\}.
$
These two excluded values follow from the reducibility condition
$2\left(s+\frac{n-1}{2}\right)=\pm1.
$

If $s\neq0$, the assumptions on $s$ imply that
$$
\mathcal Q(\mathfrak m_s^{+})
\simeq
\nu_{\mathds{1}_1}^{s+\frac{n-1}{2}}
\times
\nu^{-2s-1},
$$
which is the representation in
Conjecture~\ref{conj}\textup{(2)}.

If $s=0$, then
$$
\mathfrak m_0^{+}
=
\left\{
\left[\frac{n-1}{2}\right]_{(\mathds{1}_1)},
\left[\frac{n-3}{2}\right]_{(\mathds{1}_1)},
\dots,
\left[-\frac{n-1}{2}\right]_{(\mathds{1}_1)}
\right\}
=
\mathfrak c_n(0).
$$
Therefore,
$
\pi
\simeq
\mathcal Q(\mathfrak c_n(0))
\simeq
\mathds{1}_n.
$
This is the unique irreducible subrepresentation prescribed in
Conjecture~\ref{conj}\textup{(2)} at $s=0$.

Now suppose that $\tau=\tau_s^{-}$. Then
$
\mathfrak a_s\cup\mathfrak b
=
\mathfrak m_s^{-},
$
and hence
$
\pi\simeq\mathcal Q(\mathfrak m_s^{-}).
$
The representation $\tau_s^{-}$ is irreducible precisely when
$
s\notin
\left\{
\frac{n-2}{2},\frac n2
\right\},
$
because its reducibility condition is
$
2\left(s-\frac{n-1}{2}\right)=\pm1.
$
If $s\neq0$, we obtain
$$
\pi
\simeq
\nu_{\mathds{1}_1}^{s-\frac{n-1}{2}}
\times
\nu^{-2s+1},
$$
which is the representation in
Conjecture~\ref{conj}\textup{(3)}.
If $s=0$, then
$
\mathfrak m_0^{-}
=
\mathfrak c_n(0),
$
so that
$
\pi\simeq\mathds{1}_n.
$
This is the unique irreducible subrepresentation in
Conjecture~\ref{conj}\textup{(3)} at $s=0$.

Thus, every reducible principal-series case produces one of the two
exceptional families or the trivial representation.

\medskip

\noindent\emph{Converse Case B2: $\tau=\St_2$.}
The Steinberg representation $\St_2$ corresponds to the segment
$
\left[-\frac12,\frac12\right]_{(\mathds{1}_1)}.
$
Applying the Jacquet--Langlands transfer gives
$
\JL
\left(
\left[-\frac12,\frac12\right]_{(\mathds{1}_1)}
\right)
=
\left[-\frac32,\frac32\right]_{(\mathds{1})}.
$
Thus, we have
$
\mathfrak L(\St_2)\simeq\Sp_4.
$
Equation \eqref{5.10} becomes
$$
\mathfrak L(\pi)
\simeq
\mathfrak L(\nu^{-2s})\oplus\Sp_4.
$$
The corresponding Langlands multisegment is
$
\mathfrak a_s
\cup
\left\{
\left[-\frac12,\frac12\right]_{(\mathds{1}_1)}
\right\}.
$
By Lemma~\ref{venk},
$
\nu^{-2s}\times\St_2
$
is reducible if and only if
$
s=\pm\frac n2.
$
If $s\neq\pm\frac n2$, then the induction is irreducible, and
injectivity gives
$
\pi\simeq\nu^{-2s}\times\St_2.
$
This representation belongs to
Conjecture~\ref{conj}\textup{(1)}.

Suppose that $s=\frac n2$. Then  the corresponding Langlands multisegment is
$$
\left\{
\left[-n+\frac32\right]_{(\mathds{1}_1)},
\dots,
\left[-\frac32\right]_{(\mathds{1}_1)},
\left[-\frac12,\frac12\right]_{(\mathds{1}_1)}
\right\}.
$$
By Remark~\ref{LnZn}, this is exactly the Langlands multisegment of
$\nu^{-n}\mathcal Q_n$. Hence,
$
\mathfrak L(\pi)
\simeq
\mathfrak L(\nu^{-n}\mathcal Q_n).
$
Injectivity gives
$$
\pi
\simeq
\nu^{-n}\mathcal Q_n
=
\nu^{-2s}\mathcal Q_n.
$$
This is the unique irreducible subrepresentation occurring in
Conjecture~\ref{conj}\textup{(3)} at $s=\frac n2$.

Finally, suppose that $s=-\frac n2$. Then the corresponding Langlands multisegment is
$$
\left\{
\left[\frac32\right]_{(\mathds{1}_1)},
\dots,
\left[n-\frac32\right]_{(\mathds{1}_1)},
\left[-\frac12,\frac12\right]_{(\mathds{1}_1)}
\right\}.
$$
This is the Langlands multisegment of
$\nu^{n}\widetilde{\mathcal Q}_n$. Therefore,
$
\mathfrak L(\pi)
\simeq
\mathfrak L(\nu^{n}\widetilde{\mathcal Q}_n),
$
and injectivity gives
$$
\pi
\simeq
\nu^{n}\widetilde{\mathcal Q}_n
=
\nu^{-2s}\widetilde{\mathcal Q}_n.
$$
This is the unique irreducible subrepresentation occurring in
Conjecture~\ref{conj}\textup{(2)} at $s=-\frac n2$.

We have now shown in every case that $\pi$ belongs to one of the
families in Conjecture~\ref{conj}. Since Conjecture~\ref{conj} is
assumed, $\pi$ admits a generalized linear period with respect to
$(\Ha_{1,n-1},\chi_s)$. This proves the converse implication and
completes the proof.
\end{proof}

\subsection{Lapid--Prasad Conjecture}\label{LPC}
In this subsection, we  prove Theorem \ref{Lapid-Prasad}, thereby confirming the Lapid--Prasad conjecture for the symmetric pair $(\G_n,\Ha_{1,n-1})$; see \cite[Conjecture 2]{Prasad2015} and \cite[Conjecture 1.1]{Kapon2025}. We begin by recalling the conjecture and the
relevant notion of an $L$-packet. If $\pi$ is an irreducible smooth
representation of a connected reductive group $G$, we denote by $\Pi_\pi$ the
$L$-packet containing $\pi$, namely, the set of irreducible admissible
representations of $G$ associated with the same Langlands parameter as $\pi$. 
\begin{conj}[Lapid--Prasad]
Let $G$ be a connected reductive algebraic group defined over a local non-Archimedean field $\F$, let $\theta:G\longrightarrow G$ be an involution defined over $\F$, and let $H=G^\theta$. Let $\pi$ be a smooth irreducible representation of $G$. If $\pi$ is $H$-distinguished, then the $L$-packet of $\pi$ is invariant under the functor
$$\rho\longmapsto\widetilde{\rho}^{\,\theta},$$
where $\widetilde{\rho}$ denotes the contragredient representation and $\rho^\theta$ denotes the twist of $\rho$ by $\theta$.
\end{conj}
\begin{proof}[Proof of Theorem~\ref{Lapid-Prasad}]
For every $\pi\in\Irr(\G_n)$, the associated $L$-packet is of size one,
namely, $\Pi_\pi=\{\pi\}$. 
Therefore, it is enough to prove that $\widetilde{\pi}^{\,\theta}\simeq\pi$
for every $\Ha_{1,n-1}$-distinguished representation $\pi\in\Irr(\G_n)$.
Let $\pi$ be such a representation. By Conjecture~\ref{conj}, either
$\pi=\mathds{1}_n$, or $\pi=\mathds{1}_{n-2}\times\tau$,
where $\tau\in\Irr(\G_2)$ is infinite-dimensional and
$\Ha_{1,1}$-distinguished. Recall that $\theta$ is the involution of $\G_n$
defined by
\[
\theta(g)=\delta_{1,n-1}g\delta_{1,n-1}^{-1},
\qquad g\in\G_n,
\]
where $\delta_{1,n-1}=\diag(1,-I_{n-1})$.
We first observe that twisting by $\theta$ does not change the isomorphism class of a representation of $\G_n$. Let $\pi$ be a representation of $\G_n$ on a vector space $V_\pi$. Define
\[
T:V_\pi\longrightarrow V_\pi,\qquad
T(v)=\pi(\delta_{1,n-1})v.
\]
We claim that $T$ is an intertwining operator between $\pi$ and $\pi^\theta$.
Indeed, for every $g\in\G_n$ and $v\in V_\pi$, we have
\[
\begin{aligned}
T\bigl(\pi(g)v\bigr)
&=\pi(\delta_{1,n-1})\pi(g)v\\
&=\pi(\delta_{1,n-1}g)v\\
&=\pi(\delta_{1,n-1}g\delta_{1,n-1}^{-1})
  \pi(\delta_{1,n-1})v\\
&=\pi^\theta(g)T(v).
\end{aligned}
\]

Thus, $T$ intertwines $\pi$ and $\pi^\theta$. Moreover, it is
bijective, with inverse $T^{-1}=\pi(\delta_{1,n-1}^{-1})$. In fact, since $\delta_{1,n-1}^2=1$, we have $T^{-1}=T$.
Consequently, $T$ gives an isomorphism $\pi^\theta\simeq\pi$. In particular,
$\widetilde{\pi}^{\,\theta}\simeq\widetilde{\pi}.$
It therefore remains to show that $\pi$ is self-contragredient. If $\pi=\mathds{1}_n$, this is immediate. In the second case, the
$\Ha_{1,1}$-distinction of $\tau$ implies that
$\widetilde{\tau}\simeq\tau$.
Since taking the contragredient is compatible with parabolic induction and $\mathds{1}_{n-2}$ is self-contragredient, it follows that
\[
\widetilde{\pi}
 \simeq\widetilde{\mathds{1}_{n-2}\times\tau}
 \simeq\mathds{1}_{n-2}\times\widetilde{\tau}
 \simeq\mathds{1}_{n-2}\times\tau
 =\pi.
\]
Thus, the $L$-packet $\Pi_\pi=\{\pi\}$ is invariant under
$\rho\mapsto\widetilde{\rho}^{\,\theta}$, as required.

\end{proof}
\begin{remark}
The preceding result is specific to the case $s=0$ and does not extend, in
general, to generalized linear periods associated with nontrivial characters
$\chi_s$. Indeed, Conjecture~\ref{conj} predicts that, for $s\notin\left\{-\frac{n}{2},0\right\}$, 
the irreducible representation $\pi_s
 =\nu_{\mathds{1}_1}^{\,s+\frac{n-1}{2}}
  \times\nu^{-2s-1}$
admits a generalized linear period with respect to
$(\Ha_{1,n-1},\chi_s)$. However, $\pi_s$ is not self-contragredient.
Thus, the analogue of the Lapid--Prasad conjecture for representations
admitting a generalized linear period with respect to
$(\Ha_{1,n-1},\chi_s)$ fails for arbitrary values of $s$.
\end{remark}

\subsection*{Acknowledgement} 
The authors thank U. K. Anandavardhanan for his helpful discussion. We also thank Dipendra Prasad for his valuable comments and insightful suggestions.
\bibliographystyle{alpha}
\bibliography{linear}		
\end{document}